\documentclass[11pt]{article}
\usepackage{amssymb}
\usepackage{color}
\newcommand{\bC}{\mathbf{C}}
\newcommand{\bN}{\mathbf{N}}
\newcommand{\bR}{\mathbf{R}}
\newcommand{\bZ}{\mathbf{Z}}
\newcommand{\bK}{\mathbf{K}}
\newcommand{\ord}{\mbox{\rm ord }}
\newcommand{\ini}{\mbox{\rm in}}

\newlength{\szer}
\newcommand{\Teiss}[2]{%
\settowidth{\szer}{$\displaystyle\frac{#1}{#2}$}%
\setlength{\szer}{0.5\szer}%
\left\{\hspace{\szer}%
\raisebox{0.15ex}{\makebox[0pt]{$\displaystyle\frac{#1}{\phantom{#2}}$}}%
\raisebox{-0.15ex}{\makebox[0pt]{$\displaystyle\frac{\phantom{#1}}{#2}$}}%
\hspace{\szer}\right\}%
}

\newtheorem{defi}{Definition}[section]
\newtheorem{nota}[defi]{Remark}
\newtheorem{ejemplo}[defi]{Example}
\newtheorem{ejemplos}[defi]{Examples}

\newtheorem{teorema}[defi]{Theorem}
\newtheorem{prop}[defi]{Proposition}
\newtheorem{lema}[defi]{Lemma}
\newtheorem{coro}[defi]{Corollary}

\newtheorem{propiedad}[defi]{Property}

\newenvironment{proof}[1][Proof]{\textbf{#1.} }{\
\rule{0.5em}{0.5em}}

\begin{document}
\title{An approach to plane algebroid  branches
\footnotetext{
     \noindent   \begin{minipage}[t]{4in}
       {\small
       2000 {\it Mathematics Subject Classification:\/} Primary 32S55;
       Secondary 14H20.\\
       Key words and phrases: plane algebroid curve, branch, semigroup associated with a branch, key polynomials, logarithmic distance, Abhyankar-Moh theory.\\
       The first-named author was partially supported by the Spanish Project
       PNMTM 2007-64007.}
       \end{minipage}}}

\author{Evelia R.\ Garc\'{\i}a Barroso and Arkadiusz P\l oski}

\maketitle

\begin{abstract}
\noindent Our aim is to reprove the basic results of the theory of branches of plane
algebraic curves over algebraically closed fields of arbitrary characteristic. We do not use the Hamburger-Noether expansions. Our basic tool is the logarithmic distance on the set of branches satisfying the strong triangle inequality which permits to make calculations directly on the equations of branches.
\end{abstract}

\hfill{\it What can be explained on fewer principles}

\hfill{\it is explained needlessly by more.}

\hfill{\it William of Ockham (1280-1349)\footnote{Quoted after
Samuel E. Stumpf. Socrates to Sartre. A History of Philosophy. 
Mc \\Graw-Hill, Inc. 1993.}}

\section*{Introduction}
\label{intro}
\noindent We present a new approach to the theory of  plane algebroid branches over an algebraically closed field of arbitrary characteristic. We prove the structure theorem for the semigroup of plane branches, the fundamental theorems of the Abhyankar-Moh theory, the intersection formula and the existence of a branch with given semigroup. These results are well-known (at least in characteristic $0$) but our proofs are new. In constrast to classical treatments of the subject given by Ancochea (1947), Lejeune-Jalabert (1973), Moh (1973), Angerm\"uller (1977), Russel (1980) and Campillo (1980)  we do not use the quadratic transformations. To avoid the Hamburger-Noether expansions we base our approach on the direct construction of key polynomials (the notion introduced by MacLane (1936)) given by Seidenberg in his PhD thesis on the valuation ideals in polynomial rings. As far as we know the Seidenberg article of 1945 is the first publication in which appears the God-given inequality
$n_k\overline{\beta_k}<\overline{\beta_{k+1}}$ (we use the notation introduced by Zariski).

\medskip

\noindent In all this paper we use the strong triangle inequality (STI) proved by the second author in 1985. It allows to give simple proofs of all basic properties of key polynomials in any characteristic.
 Using the STI we prove the Abhyankar-Moh irreducibility criterion in arbitrary characteristic, the description of branches with given semigroup and the Merle-Granja factori\-zation theorem.

\medskip

\noindent A plane algebroid branch may be given either by an irreducible
equation $f(x,y)=0$ or by a parametrization $x=\phi(t)$, $y=\psi(t)$. The
treatments of the subject which use the Hambuger-Noether expansions 
(or Puiseux' expansion in the case of characteristic 0) are based on the
interplay between the equations and the parametrizations of branches. 
In
this paper after having proved the STI we make calculations on the equations of branches without 
recourse to their parametrizations.  In particular we prove a new formula for the intersection multiplicity of two
branches, which does not involve any reference to their parametrizations.
In this
way we get shorter and conceptually simpler proofs of basic theorems than
in the classical approach to plane algebroid branches.

\medskip

\noindent The contents of this article are

\begin{enumerate}
\item Preliminaries
\begin{enumerate}
\item [1.1] Arithmetical lemmas and semigroups of naturals
\item [1.2] Plane algebroid curves
\end{enumerate}
\item The strong triangle inequality
\item The semigroup of a plane algebroid branch
\item A proof of the Semigroup Theorem
\item Key polynomials
\item The Abhyankar-Moh theory
\item A formula for the intersection multiplicity of two branches
\item The Abhyankar-Moh irreducibility criterion
\item Characterization of the semigroups associated with branches
\item Description of branches with given semigroup
\item Merle-Granja's Factorization Theorem
\end{enumerate}

\medskip
\noindent The following notation is used in the sequel. The set of all integers (resp. non-negative integers) is denoted by 
$\bZ$ (resp. $\bN$). We write $\gcd S$ for the greatest common divisor of a nonempty subset $S\subset \bN$.  Conventions about calculating with $+\infty$ are usual. In all this note $\bK$ is an algebraically closed field of arbitrary characteristic.

\section{Preliminaries}

\noindent In this section we fix our notations and recall some useful notions and results.

\subsection{Arithmetical lemmas and semigroups of naturals}
\label{arithmetical-lemmas}
\noindent We recall  here some properties of semigroups of natural numbers
 that we will use in Section \ref{semigroup-section} of this paper.

\begin{lema}
\label{Bezout}
Let $v_0,\ldots,v_k$ be a sequence of positive integers. Set $d_i=\gcd (v_0,\ldots,v_i)$ for $i\in\{0,1,\ldots,k\}$ and $n_i=\frac{d_{i-1}}{d_i}$ for $i\in\{1,\ldots,k\}$. Then for every $a\in \bZ d_k$ we have B\'ezout's relation:
\[a=a_0v_0+a_1v_1+\cdots+a_kv_k,\]

\noindent where $a_0\in \bZ$ and $0\leq a_i<n_i$ for $i\in\{1,\ldots,k\}$. The sequence $(a_0,\ldots,a_k)$ is unique.
\end{lema}

\noindent \begin{proof}
\noindent \underline{Existence:} if $k=0$ the lemma is obvious. Suppose that $k>0$ and that the lemma is true for $k-1$. Since $(d_k)\bZ=(d_{k-1},v_k)\bZ$ we can write for every $a\in (d_k)\bZ$: $a=a'd_{k-1}+a^{''}v_k$ with $a',a^{''}\in \bZ$. For any integer $l$ we have $a=(a'-l v_k)d_{k-1}+(a^{''}+l d_{k-1})v_k$. Thus we can take $a^{''}\geq 0$. Dividing $a^{''}$ by $\frac{d_{k-1}}{d_k}$ we get $a^{''}=\left(\frac{d_{k-1}}{d_k}\right)a^{'''}+a_k$ with $0\leq a_k<\frac{d_{k-1}}{d_k}$. Therefore
\[a=a'd_{k-1}+\left(\frac{d_{k-1}}{d_k}a^{'''}+a_k\right)v_k=\left(a'+
\frac{v_{k}}{d_k}a^{'''}\right)d_{k-1}+a_kv_k.\]

\noindent By induction hypothesis we get $\left(a'+
\frac{v_{k}}{d_k}a^{'''}\right)d_{k-1}=a_0v_0+\cdots+a_{k-1}v_{k-1}$ with $0\leq a_i < \frac{d_{i-1}}{d_i}$ for $0<i\leq k-1$ and we are done.

\medskip

\noindent \underline{Unicity:} Suppose that $a_0v_0+\cdots+a_kv_k=a'_0v_0+\cdots+a'_kv_k$ with
$0\leq a_i,a'_i<\frac{d_{k-1}}{d_k}$ for $i>0$. Let $a_k\leq a'_k$. Then $(a'_k-a_k)v_k\equiv 0$ mod$(v_0,\ldots, v_{k-1})\bZ$ and consequently $(a'_k-a_k)v_k\equiv 0$ mod$(d_{k-1})\bZ$, which implies
$(a'_k-a_k)\left(\frac{v_{k}}{d_k}\right)\equiv 0$ mod $\frac{d_{k-1}}{d_k}\bZ$.

\medskip

\noindent Therefore $a'_k-a_k\equiv 0$ mod $\frac{d_{k-1}}{d_k}\bZ$ and $a'_k-a_k=0$ since
$0\leq a'_k-a_k<\frac{d_{k-1}}{d_k}$. Unicity follows by induction.
\end{proof}

\begin{lema}
\label{aritmetico}
With the above notations assume that $n_{i-1}v_{i-1}<v_i$ for $i\in \{2,\ldots, k\}$. Then

\begin{enumerate}
\item[(i)] $n_kv_k\in \mathbf N
v_0+\cdots+\mathbf N v_{k-1}$,
\item [(ii)] if $a\in \mathbf N
v_0+\cdots+\mathbf N v_{k}$ then there are integers $a_0,\ldots,a_k$ such that $a=a_0v_0+a_1v_1+ \cdots+a_kv_k$, where $0\leq a_0$ and $0\leq a_i<n_i$ for  $i\in \{1,\ldots,k\}$.
\end{enumerate}
\end{lema}
\noindent\begin{proof}
\begin{enumerate}
\item Since $n_kv_k=d_{k-1}\frac{v_k}{d_k}\equiv 0$ (mod $d_{k-1}\bZ$), by Lemma \ref{Bezout}  we can write B\'ezout's identity
$$n_kv_k=a_0v_0+a_1v_1+\cdots+a_{k-1}v_{k-1},$$

\noindent where $a_0\in \bZ$ and $0\leq a_i<n_i$ for $i\in \{1,\ldots,k-1\}$ .

\noindent Therefore we get

\begin{eqnarray*}
a_0v_0&=&n_kv_k-a_1v_1-\cdots-a_{k-1}v_{k-1}\\
&\geq& n_kv_k-(n_1-1)v_1-\cdots-(n_{k-1}-1)v_{k-1}\\
&=& n_kv_k-[(n_1v_1-v_1)+\cdots+(n_{k-1}v_{k-1}-v_{k-1})]\\
&>&n_kv_k-[(v_2-v_1)+\cdots+(v_k-v_{k-1})]\\
&=&n_kv_k-v_k+v_1>0
\end{eqnarray*}

\noindent which proves $(i)$.

\item We have to check that $a\in  \mathbf N
v_0+\cdots+\mathbf N v_{k}$ implies $a_0\geq 0$, where $a_0\in \bZ$ is defined by B\'ezout's identity.
If $k=0$ it is obvious. Suppose that $k>0$ and that the property is true for $k-1$. By assumption we have
$a=q_0v_0+\cdots+q_kv_k$ with $q_i\geq 0$ for $i\in \{0,\ldots,k\}$. By the Euclidean division of $q_k$ by $n_k$ we get
$q_k=q'_kn_k+a_k$ with $0\leq a_k<n_k$. Thus $a=q_0v_0+\cdots+q_{k-1}v_{k-1}+q'_kn_kv_k+a_kv_k=a'+a_kv_k$, where
$0\leq a_k<n_k$ and $a' \in \mathbf N
v_0+\cdots+\mathbf N v_{k-1}$ by Property $(i)$. Use the induction hypothesis.
 \end{enumerate}  \end{proof}

\begin{nota}
In fact we have proved the following property, stronger that the first part of Lemma \ref{aritmetico}: if
$n_kv_k=a_0v_0+a_1v_1+\cdots+a_{k-1}v_{k-1}$ is B\'ezout's relation then $a_0>0$.
\end{nota}

\begin{nota}
\label{n}
Obviously $n_k\geq 1$. From the first part of  Lemma \ref{aritmetico} it follows that $n_k>1$ if  (and only if) $\bN v_0+\cdots +\bN v_{k-1}\neq \bN v_0+\cdots +\bN v_{k}$.
\end{nota}

\medskip

\noindent Let $n>0$ be an integer. A sequence of positive integers $(v_0,\ldots,v_h)$ is said to be a {\em Seidenberg n-characteristic sequence} or {\em n-characteristic sequence} if $v_0=n$ and it satisfies the following two axioms
\begin{enumerate}
\item Set $d_i=\gcd (v_0,\ldots,v_i)$ for $0\leq i \leq  h$ and $n_i=\frac{d_{i-1}}{d_i}$ for $1\leq i\leq h$. Then $d_h=1$ and $n_i>1$ for $1\leq i \leq  h$.

\item $n_{i-1}v_{i-1}<v_i$ for $2\leq i\leq h$.
\end{enumerate}

\noindent Note that condition $(2)$ implies that the sequence $(v_1,\ldots,v_h)$ is strictly increasing. If $n> 1$ then $h\geq 1$. If $h=1$ then the sequence $(v_0,v_1)$ is a Seidenberg  $n$-characteristic sequence if and only if $v_0=n$ and $\gcd (v_0,v_1)=1$. There is exactly one $1$-sequence which is $(1)$. Note also that $2^h\leq n$.

\medskip

\noindent If $(v_0,\ldots,v_h)$ is an $n$-characteristic sequence then for any $k\in \{1,\ldots,h\}$ the sequence $\left(\frac{v_0}{d_k},\ldots, \frac{v_k}{d_k}\right)$ is an $\frac{n}{d_k}$-characteristic sequence. Its associated sequences are $\left(\frac{d_0}{d_k},\ldots, \frac{d_k}{d_k}\right)$ and $(n_1,\ldots, n_k)$.

\medskip

\noindent We say that a subset $G$ of $\bN$ is a {\em semigroup} if it contains $0$ and if it is closed under addition.

\medskip

\noindent Let $G$ be a nonzero semigroup and let $n\in G$, $n>0$. Then there exists (cf. \cite{Hefez}, Chapter 6, Proposition 6.1) a unique sequence $v_0,\ldots,v_h$ such that $v_0=n$, $v_k=\min (G \backslash v_0\bN+\cdots+v_{k-1}\bN)$  for $k\in\{1,\ldots,h\}$ and $G=v_0\bN+\cdots+v_h\bN$. We call the sequence $(v_0,\ldots,v_h)$ the $n$-{\em minimal system of generators of } $G$. If $n=\min(G\backslash\{0\})$ then we say that $(v_0,\ldots, v_h)$ is the {\em minimal set} of generators of $G$. We will study semigroups generated by  $n$-characteristic sequences.

\begin{prop}
\label{prop-min}
Let $G=v_0\bN+\cdots+v_h\bN$ where $(v_0,\ldots,v_h)$ is an $n$-characteristic sequence. Then
\begin{enumerate}
\item The sequence $(v_0,\ldots,v_h)$ is the $n$-minimal system of generators of $G$.
\item $\min (G\backslash\{0\})=\min(v_0,v_1)$.

\item The minimal system of generators of $G$ is $(v_0,v_1,\ldots,v_h)$ if $v_0<v_1$, $(v_1,v_0,\ldots,v_h)$ if $v_1<v_0$  and
$v_0\not\equiv 0$ (mod $v_1$) and $(v_1,v_2,\ldots,v_h)$ if $v_0\equiv 0$ (mod $v_1$). Moreover, the minimal system of generators of $G$ is a
$\min (G\backslash \{0\})$-characteristic sequence. 

\item Let $c=\sum_{k=1}^h(n_k-1)v_k-v_0+1$. Then for every $a,b\in \bZ$: if $a+b=c-1$ then exactly one element of the pair $(a,b)$ belongs to $G$. Consequently $c$ is the smallest element of $G$  such that all integers bigger than or equal to it are in $G$.
\item $c$ is an even number and $\sharp (\bN\backslash G)=\frac{c}{2}$.
\end{enumerate}
\end{prop}

\noindent \begin{proof}
We leave to the reader the proof of the first three claims. To prove the fourth claim (see \cite{Sathaye-Stenerson} ) we take two integers $a,b\in \bZ$ such that $a+b=c-1$. Let us write B\'ezout's relation $a=a_0v_0+a_1v_1+\cdots+a_hv_h$ where $a_0\in \bZ$ and $0\leq a_i<n_i$ for $i\in \{1,\ldots,h\}$. Then by definition of $c$ we get $b=c-1-a=-v_0+\sum_{k=1}^h(n_k-1)v_k-a_0v_0-
\sum_{k=1}^h a_k v_k=-(a_0+1)v_0+\sum_{k=1}^h (n_k-1-a_k)v_k$. This is a
B\'ezout relation. To finish the proof it suffices to remark that exactly one element of the pair
$(a_0,-a_0-1)$ is greater than or equal to zero. For the last remark, note that $(c+N)+(-N-1)=c-1$ and hence,
if $N\geq 0$, then $-N-1\not\in G$ and consequently $c+N\in G$ for all $N\geq 0$. On the other
hand, since $0\in G$ we have $c-1\not\in G$ and hence $c$ is the smallest integer such that
all integers bigger than or equal to it are in $G$. 

\medskip

\noindent Finally we will prove the fifth claim. The mapping $[0,c-1]\cap G \ni a \rightarrow c-1-a\in [0,c-1]\cap (\bN\backslash G)$ is bijective. Therefore we have $2 \cdot \sharp ([0,c-1]\cap G)=c$ and the last claim follows.
\end{proof}

\medskip

\noindent The number $c$ is called the {\em conductor} of the semigroup $G$.

\subsection{Plane algebroid curves}

\noindent We review here some basic notions from the local theory of algebraic curves. For more details we refer the reader to \cite{Seidenberg2}.

\medskip

\noindent Let $f\in \bK[[x,y]]$ be a non-zero power series without constant term. An algebroid curve $\{f=0\}$ is defined to be the ideal generated by
$f$ in $\bK[[x,y]]$. We say that $\{f=0\}$ is irreducible (reduced) if $f$ in $\bK[[x,y]]$ is irreducible ($f$ has no multiple factors). The irreducible curves are also called {\em branches}. The order $\ord f$ of the power series $f$ is, by definition, the {\em multiplicity} of the curve $\{f=0\}$. The {\em initial form} $\ini f$ of $f$ defines the tangent lines of $\{f=0\}$. If $\{f=0\}$ is irreducible then it has only one tangent line i.e. $\ini f=l^{\ord f}$ where $l$ is a linear form.

\medskip

\noindent A {\em formal isomorphism} $\Phi$ is a pair of power series $\Phi(x,y)=(ax+by+\cdots, a'x+b'y+\cdots)$ where $ab'-a'b\neq 0$ and the dots denote terms in $x,y$ of order bigger than 1. The map $f\longrightarrow f\circ \Phi$ is an isomorphism of the ring $\bK[[x,y]]$. Two curves $\{f=0\}$ and $\{g=0\}$ are said to {\em{ be formally equivalent}} if there is a formal isomorphism $\Phi$ such that $f\circ \Phi=g \cdot \mathrm{unit}$.

\medskip

\noindent For any power series $f,g\in \bK[[x,y]]$ we define the {\em intersection multiplicity or intersection number} $i_0(f,g)$ by putting
\[i_0(f,g)=\mathrm{dim}_{\bK}\bK[[x,y]]/(f,g), \]

\noindent where $(f,g)$ is the ideal of $\bK[[x,y]]$ generated by $f$ and $g$. If $f,g$ are non-zero power series without constant term then $i_0(f,g)<+\infty$ if and only if $\{f=0\}$ and $\{g=0\}$ have no common branch. The following properties are basic

\begin{enumerate}
\item if $\Phi$ is a formal isomorphism then $i_0(f,g)=i_0(f\circ \Phi, g\circ \Phi)$.
\item $i_0(f,gh)=i_0(f,g)+i_0(f,h)$.
\end{enumerate}

\noindent Let $t$ be a variable. A {\em parametrization} is a pair $(\phi(t),\psi(t))\in \bK[[t]]^2$ such that $\phi(t)\neq 0$ or $\psi(t)\neq 0$ in $\bK[[t]]$ and $\phi(0)=\psi(0)=0$. We say that the parametrization $(\phi(t),\psi(t))$ is {\em{good}} if the field of fractions of the ring $\bK[[\phi(t),\psi(t)]]$ is equal to the field $\bK((t))$.

\begin{teorema}[Normalization Theorem]
Let $f=f(x,y)\in \bK [[x,y]]$ be an irreducible power series. Then there is a good parametrization $(\phi(t),\psi(t))$ such that $f(\phi(t), \psi(t))=0$. If $(\alpha(s), \beta(s))\in \bK [[s]]^2$ is a parametrization such that $f(\alpha(s),\beta(s))=0$ then there is a power series $\sigma(s)\in \bK [[s]]$, $\sigma(0)=0$ such that $\alpha(s)=\phi(\sigma(s))$ and $\beta(s)=\psi(\sigma(s))$.
\end{teorema}

\noindent Let us recall also
\begin{teorema}
\label{classical-formula}
Under the above assumptions and notations, for any power series $g=g(x,y)\in \bK[[x,y]]$ we have $i_0(f,g)=\ord g(\phi(t),\psi(t))$.
\end{teorema}

\noindent Taking $g=x$  (respect. $g=y$) we get from the above formula that $\ord f(0,y)=i_0(f,x)=\ord \phi(t)$ and $\ord f(x,0)=i_0(f,y)=\ord \psi(t)$.

\medskip

\noindent Using  Theorem \ref{classical-formula} we check the following two properties of intersection numbers:

\begin{enumerate}
\item [3.] If  $f$ is irreducible, then $i_0(f,g+g')\geq \inf \{i_0(f,g),i_0(f,g')\}$ with equa\-lity if $i_0(f,g)\neq i_0(f,g')$.
\item [4.] If $f$ is irreducible and $i_0(f,g)=i_0(f,h)<+\infty$ then there exists a constant $c\in \bK$ such that $i_0(f,g-ch)>i_0(f,g)$.
\end{enumerate}

\medskip

\noindent In what follows we need

\begin{lema}
\label{ordenes}
Let $f(x,y)\in \bK[[x,y]]$ be an irreducible power
series such that $f(0,y)\neq 0$ and let $(\alpha(s),\beta(s))$, $\alpha(s)\neq 0$ in
$\bK[[s]]$, be a parametrization such that
$f(\alpha(s),\beta(s))=0$. Then, for every power series $g(x,y)\in
\bK[[x,y]]$ we have
\[\ord g(\alpha(s),\beta(s))=\frac{i_0(f,g)}{i_0(f,x)}\ord
\alpha(s).\]
\end{lema}

\noindent \begin{proof} \noindent Let $(\phi(t),\psi(t))$ be a good parametrization of
the branch $\{f(x,y)=0\}$. Then $\alpha(s)=\phi(\sigma(s))$,
$\beta(s)=\psi(\sigma(s))$ for a power series $\sigma(s)\in \bK[[s]]$,
$\sigma(0)=0$. We get
$\ord \alpha(s)=\ord \phi(t)\ord \sigma(s)=\ord f(0,y)\ord \sigma(s)=
i_0(f,x)\ord \sigma(s)$

\noindent and consequently

\[ \ord \sigma(s)=\frac{\ord \alpha(s)}{i_0(f,x)}.
\]

\noindent On the other hand $\ord g(\alpha(s),\beta(s))=\ord g(\phi(t),\phi(t)).
\ord \sigma(s)$ and by Theorem \ref{classical-formula} we get
\[
\ord g(\alpha(s),\beta(s))=i_0(f,g) \ord \sigma(s).
\]

\noindent Now the
formula for $\ord g(\alpha(s),\beta(s))$ follows.
\end{proof}

\medskip

\noindent For any irreducible power series $f\in \bK[[x,y]]$ we put
\[\Gamma(f)=\{i_0(f,g)\;:\; g\; \hbox{\rm runs over all power series such
that } g\not\equiv 0 \;\hbox{\rm (mod $f$)}  \}.\]

\noindent Clearly $\Gamma(f)$ is a semigroup. We call $\Gamma(f)$ the {\em semigroup associated with the branch} $\{f=0\}$. 

\medskip

\noindent Two branches $\{f=0\}$ and $\{g=0\}$ are {\em equisingular} if and only if $\Gamma(f)=\Gamma(g)$. Two formally equivalent branches are equisingular. The branch $\{f=0\}$ is {\em non-singular} (that is of multiplicity $1$) if and only if $\Gamma(f)=\bN$. We have $\min (\Gamma(f)\backslash\{0\})=\ord f$.

\medskip

\noindent Different (but equivalent) definitions of equisingularity were given by Zariski in \cite{Zariski-1965}.

\medskip

\noindent Note that the mapping $g\mapsto i_0(f,g)$ induces a valuation $v_f$ of the ring $\bK[[x,y]]/(f)$. The semigroup $\Gamma(f)$ can be described as the semigroup of values of $v_f$.

\section{The strong triangle inequality}

\noindent In this section we generalize the well-known property of the intersection multiplicity \cite{Ploski} to any characteristic.  Let us begin with the notion of logarithmic distance.

\noindent Let $A$ be a non-empty set. A function $d:A\times
A\longrightarrow \mathbf R\cup\{+\infty\}$ satisfying for arbitrary
$a,b,c\in A$, the conditions:
\begin{enumerate}
\item[(i)] $d(a,a)=+\infty$,
\item[(ii)] $d(a,b)=d(b,a)$,
\item[(iii)] $d(a,b)\geq \hbox{\rm inf}\{d(a,c),d(b,c)\}$,
\end{enumerate}

\noindent will be called a {\em logarithmic distance} (for short
log-distance). We call the third property the {\em Strong Triangle
Inequality} (the STI). It is equivalent to the following

\medskip

\noindent (iii') at least two of the numbers $d(a,b),d(a,c),d(b,c)$
are equal and the third is not smaller than the other two.

\begin{lema}
\label{elementos} Let $d$ be a log-distance in the set $A$. For any
$a_1,\ldots,a_m$, $b_1,\ldots,b_n,c\in A$ at least one of the
following conditions holds:
\begin{enumerate}
\item[\hbox{\rm (I)}] there exists a $j\in\{1,\ldots,n\}$ such that for any
$i\in\{1,\ldots,m\}$, $d(a_i,c)\leq d(a_i,b_j)$,
\item[\hbox{\rm (II)}] there exists an $i\in\{1,\ldots,m\}$ such that for any
$j\in\{1,\ldots,n\}$, $d(b_j,c)\leq d(a_i,b_j)$.
\end{enumerate}
\end{lema}

\noindent \begin{proof} Let us suppose that neither (I) nor (II) holds. Then,
for any $j\in\{1,\ldots,n\}$ there exists an index $p(j)\in
\{1,\ldots,m\}$ such that $d(a_{p(j)},c)> d(a_{p(j)},b_j)$ and, for
any $i\in\{1,\ldots,m\}$, there exists $s(i)\in \{1,\ldots,n\}$
such that $d(b_{s(i)},c)> d(a_{i},b_{s(i)})$. Applying the
STI to $a_{p(j)},b_j,c$ and to $a_i,b_{s(i)},c$ we get

\begin{equation}
\label{formula1}
d(a_{p(j)},b_j)=d(b_j,c)<d(a_{p(j)},c),
\end{equation}

\noindent and

\begin{equation}
\label{formula2}
d(a_i,b_{s(i)})=d(a_i,c)<d(b_{s(i)},c).
\end{equation}

\noindent We may assume without loss of generality that

\begin{equation}
\label{formula3} d(a_{p(1)},b_1)=\displaystyle \sup_{j=1}^n\{d(a_{p(j)},b_j)\}.
\end{equation}

\noindent Using successively (\ref{formula1}), (\ref{formula2}) and
again (\ref{formula1}), we get
$$ d(a_{p(1)},b_1)<d(a_{p(1)},c)=d(a_{p(1)},b_{s(p(1))})<
 d(b_{s(p(1))},c)=d(a_{p(j_1)},b_{j_1})   $$

\noindent with $j_1=s(p(1))$. Thus we have
$d(a_{p(1)},b_1)<d(a_{p(j_1)},b_{j_1})$, which contradicts
assumption (\ref{formula3}).
\end{proof}

\medskip
\noindent An important log-distance on the set of branches can be defined by means of the intersection multiplicity. Let $\{l=0\}$ be a smooth branch. For any branches $\{f=0\}$ and $\{g=0\}$ different from the branch $\{l=0\}$ we put

\[
d_l(f,g)=\frac{i_0(f,g)}{i_0(f,l)i_0(g,l)}.
\]

\noindent Our aim is to prove

\begin{teorema}
The function $d_l$ is a log-distance in the set of all branches different from $\{l=0\}$.
\end{teorema}

\noindent \begin{proof}(cf. \cite{Chadzynski-Ploski})

\noindent We may assume $l=x$. Since $d_x(f,f)=+\infty$ and $d_x(f,g)=d_x(g,f)$ it suffices to check the STI.
Let $\{f=0\}$, $\{g=0\}$ and $\{h=0\}$ be three branches different from $\{x=0\}$. Let $m=i_0(f,x)=\ord f(0,y)$, $n=i_0(g,x)=\ord g(0,y)$, $p=i_0(h,x)=\ord h(0,y)$. Using the Weierstrass preparation theorem we may
assume that
$f,g,h$ are distinguished polynomials of degree $m,n,p$
respectively. Using the Normalization Theorem we check (see
\cite{Seidenberg2}, Theorem 21.18) that  there exist power series
$\alpha(s)$, $\alpha_i(s)$, $\beta_j(s)$ and $\gamma_k(s)$ such that
$ f(\alpha(s),y)=\prod_{i=1}^m(y-\alpha_i(s))$,
$g(\alpha(s),y)=\prod_{j=1}^n(y-\beta_j(s))$ and
$h(\alpha(s),y)=\prod_{k=1}^p(y-\gamma_k(s))$.

\medskip

\noindent The function $d:\bK[[s]]\times \bK[[s]]\longrightarrow
\mathbf R\cup \{+\infty\}$ given by $d(\alpha(s),\beta(s))=\hbox{\rm
ord}(\alpha(s)-\beta(s))$ is a log-distance in $\bK[[s]]$. Fix
$k\in\{1,\ldots,p\}$ and use Lemma \ref{elementos} to
$\alpha_1(s),\ldots,\alpha_m(s)$, $\beta_1(s),\ldots,\beta_n(s)$ and
$\gamma(s)=\gamma_k(s)$. Then

\begin{enumerate}
\item[(I)] there exists a $j\in\{1,\ldots,n\}$ such that $\ord (\alpha_i(s)-\gamma(s))\leq
\ord (\alpha_i(s)-\beta_j(s))$ for all $i\in\{1,\ldots,m\}$, or

\item[(II)] there exists an $i\in\{1,\ldots,m\}$ such that $\ord (\beta_j(s)-\gamma(s))\leq \ord (\alpha_i(s)-\beta_j(s))$, for all $j\in\{1,\ldots,n\}$.
\end{enumerate}

\medskip

\noindent If (I) holds then $\sum_{i=1}^m \ord (\alpha_i(s)-\gamma(s))
\leq \sum_{i=1}^m \ord (\alpha_i(s)-\beta_j(s))$ that is $\ord f(\alpha(s),
\gamma(s))\leq \ord f(\alpha(s),\beta_j(s))$. By Lemma \ref{ordenes} we get
$\frac{i_0(h,f)}{i_0(x,h)}\leq \frac{i_0(g,f)}{i_0(x,g)}$ which implies
$d_l(f,h)\leq d_l(f,g)$.

\medskip

\noindent If (II) holds then $\sum_{j=1}^n \ord (\beta_j(s)-\gamma(s))\leq
\sum_{j=1}^n \ord (\alpha_i(s)-\beta_j(s))$ that is $\ord g(\alpha(s),\gamma(s))
\leq \ord g(\alpha(s),\alpha_i(s))$ and again by Lemma \ref{ordenes} we get
$\frac{i_0(h,g)}{i_0(x,h)}\leq \frac{i_0(f,g)}{i_0(x,f)}$ which implies
$d_l(g,h)\leq d_l(f,g)$.

\medskip
\noindent Consequently $d_l(f,g)\geq \inf \{d_l(f,h),d_l(g,h)\}$.
\end{proof}

\medskip

\begin{coro}
The function $d(f,g)=\frac{i_0(f,g)}{\ord f\;\;\ord g}$ is a log-distance
in the set of all branches.
\end{coro}

\section{The semigroup of a plane algebroid branch}

\label{semigroup}
\label{semigroup-section}
\noindent The aim of this section is to study the structure of the semigroup associated with a plane branch. We follow the method developped by Seidenberg in \cite{Seidenberg}.

\medskip

\noindent Let $f=f(x,y)\in \bK[[x,y]]$ be an irreducible power series
and let  $\Gamma(f)$ be the semigroup associated with the branch $\{f=0\}$. Suppose that $\{f=0\}\neq \{x=0\}$ and put $n=i_0(f,x)$.
Let $(\overline{b_0},\ldots,\overline{b_h})$, $\overline{b_0}=n$ be the $n$-minimal system of generators of $\Gamma(f)$.

\begin{lema}
\label{complementario}
$\Gamma(f)$ is a numerical semigroup i.e. $\gcd(\Gamma(f))=1$.
\end{lema}

\noindent \begin{proof}
\noindent Let $(\phi(t),\psi(t))$ be a good parametrization of the branch $f(x,y)=0$. Then we have  $\bK((t))=\bK ((\phi(t),\psi(t)))$ and we can write $t=\frac{p(\phi(t),\psi(t))}{q(\phi(t),\psi(t))}$ for some $p(x,y),q(x,y)\in \bK[[x,y]]$, $q\not\equiv 0$ (mod $f$). Taking orders gives
$1=i_0(f,p)-i_0(f,q)$. Put $a:=i_0(f,p)$ and $b:=i_0(f,q)$. Then $a,b\in \Gamma(f)$ and $\gcd(a,b)=1$, which proves the lemma.
\end{proof}

\medskip

\noindent We put $e_0=n$, $e_k=\gcd(e_{k-1},  \overline{b_{k}})$ for $k\in \{1,\ldots,h\}$ and
$n_k=\frac{e_{k-1}}{e_k}$  for $k\in \{1,\ldots,h\}$.
By Lemma  \ref{complementario} we have $e_h=1$. 
In what follows we write $v_f(g)$ instead of $i_0(f,g)$.

\begin{teorema}[Semigroup Theorem]
\label{structure}
Let $\{f=0\}$ be a branch such that  $\{f=0\}\neq \{x=0\}$. Set $n=v_f(x)$ and let $\overline{b_0},\ldots,\overline{b_h}$ be
the $n$-minimal system of generators of the semigroup $\Gamma(f)$.
There exists a sequence of monic
polynomials $f_0,f_1,\ldots,f_{h-1}\in \bK [[x]][y]$ such that for  $k\in \{1,\ldots,h\}$:
\begin{enumerate}
\item[\hbox{\rm ($a_k$)}] $\hbox{\rm deg}_y(f_{k-1})=\frac{n}{e_{k-1}}$,
\item [\hbox{\rm ($b_k$)}] $v_f(f_{k-1})=\overline{b_{k}}$
\noindent for $k\in \{1,\ldots,h\}$,

\item [\hbox{\rm ($c_k$)}] if $k>1$ then $n_{k-1}\overline{b_{k-1}}<\overline{b_{k}}$.
\end{enumerate}

\noindent Moreover $n_k>1$ for all $k\in\{1,\ldots, h\}$.
\end{teorema}

\noindent We give the proof of the Semigroup Theorem in Section \ref{proof-structure}.
The sequence $\overline{b_0},\ldots,\overline{b_h}$ is a Seidenberg $n$-characteristic
sequence and  will be called the {\em Seidenberg}  $n$-{\em characteristic}
of the branch $\{f=0\}$ (with respect to the regular branch  $\{x=0\}$). We will write  $\overline{\hbox {\rm char}_x}
f=(\overline{b_0},\ldots,\overline{b_h})$. Therefore  $\overline{\hbox {\rm char}_x}f$ is
determined by $n=v_f(x)$ and the
semigroup $\Gamma(f)$. Let $f_h$ be the distinguished polynomial associated with $f$
and let $\overline{b_{h+1}}=+\infty$. Then $\deg_y f_h=\frac{n}{e_h}=n$ and
$v_f(f_h)=\overline{b_{h+1}}=+\infty.$
The polynomials $f_0,f_1,\ldots,f_{h}\in \bK [[x]][y]$ will be called {\em key polynomials} of $f$. They
are not uniquely determined by $f$.

\begin{coro}
Suppose that two branches $\{f=0\}$ and $\{g=0\}$ intersect the axis
$\{x=0\}$ with the same multiplicity $n<+\infty$. Then
$\overline{\hbox{\rm char}}_x f=\overline{\hbox{\rm char}}_x g$ if and
only if $\{f=0\}$ and $\{g=0\}$ are equisingular.
\end{coro}

\medskip

\noindent Let $\overline{\beta_0},\ldots,\overline{\beta_g}$ be the minimal
system of generators of the semigroup $\Gamma(f)$ ($\overline{\beta_0}=\min \{\Gamma(f)\backslash\{0\}\}=\ord f$).
We put $\overline{\hbox{\rm char}}f=(\overline{\beta_0},\ldots,\overline{\beta_g}).$
Note that $\overline{\hbox{\rm char}}f=\overline{\hbox{\rm char}}_x f$ if and only if $v_f(x)=\ord f$. 

\begin{coro}[Inversion formulae]
Let $\overline{\hbox{\rm char}}_x f=(\overline{b_0},\overline{b_1},\ldots,\overline{b_h})$. Then $\overline{\hbox{\rm char}}f=\overline{\hbox{\rm char}}_x f$ if and only if $\overline{b_0}<\overline{b_1}$. If $\overline{b_1}<\overline{b_0}$ and $\overline{b_0}\not\equiv 0$ (mod
$\overline{b_1})$ then $\overline{\hbox{\rm char}}f=(\overline{b_1},\overline{b_0},\ldots,\overline{b_h})$. If $\overline{b_0}\equiv 0$ (mod
$\overline{b_1})$ then $\overline{\hbox{\rm char}}f=(\overline{b_1},\overline{b_2},\ldots,\overline{b_h})$.
\end{coro}

\noindent \begin{proof}
The corollary follows from the Semigroup Theorem and from the third claim of Proposition \ref{prop-min}.
\end{proof}

\medskip

\medskip

\noindent Let  $\overline{\cal O}$  be the normalization of the ring ${\cal O}=\bK[[x,y]]/(f)$ and let ${\cal C}$ be
 the conductor ideal of  $\overline{\cal O}$  in  ${\cal O}$. 
Put $c(f)=\dim_{\bK}  \overline{\cal O}/ {\cal C}$.
Then $c(f)$ is the smallest element of $\Gamma(f)$ such that $c(f)+N\in \Gamma(f)$ for any integer $N\geq 0$ (see \cite{Campillo-libro}, p. 136).

\begin{coro}[Conductor formula]
\label{ST2}
If $\overline{\hbox{\rm char}}_x f=(\overline{b_0},\overline{b_1},\ldots,\overline{b_h})$ then \\ $c(f)=\sum_{k=1}^h
(n_k-1)\overline{b_k}-\overline{b_0}+1$.
\end{coro}

\noindent \begin{proof}
Use the Semigroup Theorem and  the fourth claim of Proposition \ref{prop-min}. 
\end{proof}

\medskip
\noindent {\bf Notes}

\noindent Seidenberg gave in \cite{Seidenberg} the description of the semigroup of a zero-dimensio\-nal valuation of
the extension $\bK(x,y)/\bK$ (\cite{Seidenberg}, Theorem 6, p. 398) in terms of generators. The case of the semigroup associated with an algebroid plane branch was studied by Azevedo in \cite{Azevedo}. His method based on the Ap\`ery sequences was extended by Angerm\"uller in \cite{Angermuller} to the case of arbitrary characteristic. For different characterizations of the numerical semigroups we refer  the reader to \cite{Hefez}, Chapter 6.

\medskip

\noindent If $n=v_f(x)\not\equiv 0$ (mod char $\bK$) the Puiseux series are available. Zariski in \cite{Zariski-1986} (see also
\cite{G-P2}, \cite{Popescu}) constructed the sequence $\overline{\beta_0},\ldots,\overline{\beta_g}$ and the corresponding sequence of key polynomials by using Puiseux series expansion determined by the equation
$f(x,y)=0$. This method turned out efficient when applied to the semigroups of integers associated with meromorphic curves (see 
\cite{Abhyankar-Moh}, \cite{Abhyankar3}). A proof of the Semigroup Theorem based on the Hamburger-Noether expansion was given by Russel in \cite{Russell} and Campillo in \cite{Campillo-libro}, \cite{Campillo}. To describe classes of equisingular plane algebroid branches one uses characteristic pairs (see \cite{Moh},\cite{LJ}).

\section{A proof of the Semigroup Theorem}
\label{proof-structure}

\noindent Let $\{f=0\}$ be a branch such that $n=i_0(f,x)<+\infty$ and let
$\overline{b_0},\ldots,\overline{b_h}$ be the $n$-minimal system of generators of the semigroup $\Gamma(f)$. Observe that by the Weierstrass Division Theorem:

\[\Gamma(f)=\{v_f(g)\;:\;g\in \bK[[x]][y]\backslash\{0\}\;:\;\deg_y g<n\}.\]

\begin{prop}
\label{maximal}
There exists a monic polynomial $f_0\in \bK [[x]][y]$  such that
\begin{enumerate}
\item[\hbox{\rm ($a_1$)}] $\hbox{\rm deg}_y(f_{0})=\frac{n}{e_{0}}=1$,
\item [\hbox{\rm ($b_1$)}] $v_f(f_{0})=\overline{b_{1}}$.
\end{enumerate}
\end{prop}

\noindent To prove Proposition \ref{maximal} we check the following three properties:
\begin{lema}[Property I$_0$]
If $\psi$ is a non-zero polynomial with $\deg_y \psi<1$ then $v_f(\psi)\in \bN \overline{b_{0}}$.
\end{lema}

\noindent \begin{proof}
Obviously $\psi  \in \bK [[x]]$. Thus $v_f(\psi)=(\ord \psi)v_f(x)\in \bN \overline{b_{0}}$.
\end{proof}

\begin{lema}[Property II$_0$]
If $\deg_y \psi<1$ then $v_f(y+\psi)\leq \overline{b_{1}}$.
\end{lema}

\noindent \begin{proof}
Let $g\in \bK [[x]][y]$ be such that $v_f(g)= \overline{b_{1}}$. By the Euclidean division we get
$g=Q \cdot (y+\psi)+\psi_1$ with $\psi_1 \in  \bK [[x]]$. Clearly $v_f(g)\neq v_f(\psi_1)$ and we get
$\overline{b_{1}}\geq \inf \{v_f(g),v_f(\psi_1)\}=v_f(g-\psi_1)=v_f(Q \cdot (y+\psi))\geq v_f(y+\psi)$.
\end{proof}

\begin{lema}[Property III$_0$]
If $\psi\in \bK [[x]]$ and $v_f(y+\psi)\in  \bN \overline{b_{0}}$ then there exists a power series
$\overline{\psi}\in \bK[[x]]$
such that $v_f(y+\overline{\psi})>v_f(y+\psi)$.
\end{lema}

\noindent \begin{proof}
There exists an integer $a\geq 0$ such that $v_f(y+\psi)=a \overline{b_{0}}=v_f(x^a)$. Therefore there  is an element $c\in \bK$
such that $v_f(y+\psi-cx^a)>v_f(y+\psi)$. We put  $\overline{\psi}=\psi-cx^a$. \end{proof}

\medskip

\noindent \begin{proof}[Proof of Proposition \ref{maximal}] From Properties (II$_0$) and (III$_0$) it follows that there exists
a monic polynomial $f_0$ of degree $1$ such that $v_f(f_0)\not\in  \bN \overline{b_{0}}$. By definition of $\overline{b_{1}}$
we get $v_f(f_0)\geq \overline{b_{1}}$. The equality follows from  Property (II$_0$).
\end{proof}

\begin{prop}
\label{intermedio}
Suppose that there exist monic polynomials $f_0,f_1,\ldots,f_{k-1}$
in $\bK[[x]][y]$ such that
\begin{enumerate}
\item[\hbox{\rm ($a_i$)}] $\hbox{\rm deg}_y(f_{i-1})=\frac{n}{e_{i-1}}$,
\item [\hbox{\rm ($b_i$)}] $v_f(f_{i-1})=\overline{b_{i}}$
\noindent for $ i\in \{1,\ldots,k\}$,

\item [\hbox{\rm ($c_i$)}] $n_{i-1}\overline{b_{i-1}}<\overline{b_{i}}$
\noindent for $ i \in \{2,\ldots,k\}$.
\end{enumerate}

\noindent Then there exists a monic polynomial $f_k\in \bK[[x]][y]$ such that

\begin{enumerate}
\item[\hbox{\rm ($a_{k+1}$)}] $\hbox{\rm deg}_y(f_{k})=\frac{n}{e_{k}}$,
\item [\hbox{\rm ($b_{k+1}$)}] $v_f(f_{k})=\overline{b_{k+1}}$,
\item [\hbox{\rm ($c_{k+1}$)}] $n_{k}\overline{b_{k}}<\overline{b_{k+1}}$.
\end{enumerate}
 \end{prop}

\noindent To prove Proposition \ref{intermedio} we check the following three properties:
\begin{lema}[Property I$_k$]
\label{Ik}
If $\psi$ is a non-zero polynomial with $\deg_y \psi<\frac{n}{e_k}$ then $v_f(\psi)\in  \mathbf N
\overline{b_0}+\cdots+\mathbf N \overline{b_k}$.
\end{lema}

\noindent \begin{proof} Let $l\leq k$.
We will prove that for  $\deg_y \psi<\frac{n}{e_l}$ we have $v_f(\psi)\in
\mathbf N \overline{b_0}+\cdots+\mathbf N \overline{b_l}$.  We proceed by induction on $l$.
The case $l=0$ is already proved (see Property I$_0$). Let
$l>0$ and suppose the property holds  for
polynomials of degree less than $\frac{n}{e_{l-1}}$. Fix $\psi \in
\mathbf K[[x]][y]$ with $\hbox{\rm deg}_y(\psi)<\frac{n}{e_{l}}$ and
consider the $f_{l-1}$-adic expansion of $\psi$:

\begin{equation}
\label{semigrupo1} \psi=\psi_0 f_{l-1}^s+\psi_1
f_{l-1}^{s-1}+\cdots+\psi_s,
\end{equation}

\noindent where $\psi_0\neq 0$, $\hbox{\rm deg}_y(\psi_i)<\hbox{\rm
deg}_y(f_{l-1})=\frac{n}{e_{l-1}}$.

\medskip

\noindent Note that $s\leq \frac{\hbox{\rm deg}_y(\psi)}{\hbox{\rm
deg}_y(f_{l-1})}<n_{l}$. Let $I$ be the set of all $i\in
\{0,\ldots,s\}$ such that $\psi_i\neq 0$. Therefore, by the
induction hypothesis we get $v_f(\psi_i)\in \mathbf N
\overline{b_0}+\cdots+\mathbf N \overline{b_{l-1}}$, and

\begin{equation}
\label{semigrupo2} v_f(\psi_i)\equiv 0 \;\;\hbox{\rm mod }e_{l-1}
\;\;\hbox{\rm for }i\in I.
\end{equation}

\noindent Moreover

\begin{equation}
\label{semigrupo3} v_f(\psi_i f_{l-1}^{s-i})\neq v_f(\psi_j
f_{l-1}^{s-j}) \;\;\hbox{\rm for } i\neq j\in I.
\end{equation}

\medskip

\noindent Indeed, suppose that (\ref{semigrupo3}) is not true, so
there exist $i,j\in I$ such that $i<j$ and $v_f(\psi_i
f_{l-1}^{s-i})=v_f(\psi_j f_{l-1}^{s-j})$. Therefore
$v_f(\psi_i)+(s-i)v_f(f_{l-1})=v_f(\psi_j)+(s-j)v_f(f_{l-1})$
and $(j-i)\overline{b_l}=v_f(\psi_j)-v_f(\psi_i)\equiv 0$ mod
$e_{l-1}$ by (\ref{semigrupo2}). The last relation implies
$(j-i)\frac{\overline{b_l}}{e_l}\equiv 0$ mod $n_l$ and consequently
$j-i\equiv 0$ mod $n_l$ because $\frac{\overline{b_l}}{e_l}$ and
$n_l$ are co-prime. We get a contradiction because $0<j-i\leq
s<n_l$. Now by (\ref{semigrupo1}) and (\ref{semigrupo3}) we get

\begin{eqnarray*}
v_f(\psi)&=&\hbox{\rm min}_{i=0}^s v_f(\psi_i
f_{l-1}^{s-i})=v_f(\psi_j f_{l-1}^{s-j})\\
& = & v_f(\psi_j)+(s-j)\overline{b_l}\in \mathbf N
\overline{b_0}+\cdots+\mathbf N \overline{b_{l}},
\end{eqnarray*}

\noindent for some $j\in I$.
\end{proof}

\begin{lema}[Property II$_k$]
\label{IIk}
If $\deg_y \psi<\frac{n}{e_k}$ then $v_f(y^{\frac{n}{e_k}}+\psi)\leq \overline{b_{k+1}}$.
\end{lema}

\noindent \begin{proof}
Let $g\in \bK [[x]][y]$ be such that $v_f(g)= \overline{b_{k+1}}$. By the Euclidean division we get
$g=Q \cdot (y^{\frac{n}{e_k}}+\psi)+\psi_1$ with $\psi_1 \in  \bK [[x]][y]$ and  $\deg_y \psi_1<\frac{n}{e_k}$.
We may assume $\psi_1\neq 0$.
Therefore $v_f(\psi_1)\in \mathbf N
\overline{b_0}+\cdots+\mathbf N \overline{b_{k}}$ by Property (I$_k$) and
 $v_f(g)=\overline{b_{k+1}}\neq v_f(\psi_1)$.
Now  we get
$\overline{b_{k+1}}\geq \inf \{v_f(g),v_f(\psi_1)\}=v_f(g-\psi_1)=v_f(Q \cdot (y^{\frac{n}{e_k}}+\psi))\geq v_f(y^{\frac{n}{e_k}}+\psi)$.
\end{proof}

\begin{lema}[Property III$_k$]
\label{IIIk}
If $\psi\in \bK [[x]][y]$ with $\deg_y \psi < \frac{n}{e_k}$ and $v_f(y^{\frac{n}{e_k}}+\psi)\in \mathbf N
\overline{b_0}+\cdots+\mathbf N \overline{b_{k}}$ then there is a polynomial $\overline{\psi}\in \bK [[x]][y]$,   $\deg_y \overline{\psi}
 < \frac{n}{e_k}$
such that $v_f(y^{\frac{n}{e_k}}+\overline{\psi})>v_f(y^{\frac{n}{e_k}}+\psi)$.
\end{lema}

\noindent \begin{proof} By Lemma \ref{aritmetico} any element of the semigroup $ \mathbf N
\overline{b_0}+\cdots+\mathbf N \overline{b_{k}}$ has the form $ a_0
\overline{b_0}+ a_1
\overline{b_1}+\cdots+a_k\overline{b_{k}}$ with $a_0\geq 0$ and $0\leq a_i < n_i$ for $i\in \{1,\ldots,k\}$. Therefore
we can write $v_f(y^{\frac{n}{e_k}}+{\psi})=v_f(x^{a_0}f_0^{a_1}\cdots f_{k-1}^{a_k})$ and
there  is an element $c\in \bK$
such that $v_f(y^{\frac{n}{e_k}}+{\psi}-cx^{a_0}f_0^{a_1}\cdots f_{k-1}^{a_k})>v_f(y^{\frac{n}{e_k}}+{\psi})$.  Let   $\overline{\psi}=\psi-cx^{a_0}f_0^{a_1}\cdots f_{k-1}^{a_k}$.
Then we have $v_f(y^{\frac{n}{e_k}}+\overline{\psi})>v_f(y^{\frac{n}{e_k}}+\psi)$. Since
$\deg_y(x^{a_0}f_0^{a_1}\cdots f_{k-1}^{a_k})=a_1+a_2\frac{n}{e_1}+\cdots+({a_k})\frac{n}{e_{k-1}}\leq
(n_1-1)+(n_2-1)\frac{n}{e_1}+\cdots+(n_k-1)\frac{n}{e_{k-1}}=\frac{n{n_k}}{e_{k-1}}-1<\frac{n}{e_{{k}}}$, 
$\deg_y \overline{\psi}<\frac{n}{e_k}$.
\end{proof}

\medskip

\noindent \begin{proof}[Proof of Proposition \ref{intermedio}] From Properties (II$_k$) and (III$_k$) it follows that there exists
a monic polynomial $f_k$ of degree $\frac{n}{e_k}$ such that $v_f(f_k)\not\in \mathbf N
\overline{b_0}+\cdots+\mathbf N \overline{b_{k}}$. By definition of $\overline{b_{k+1}}$
we get $v_f(f_k)\geq \overline{b_{k+1}}$. The equality follows from  Property (II$_k$).

\medskip

\noindent To check ($c_{k+1}$) observe that $v_f(f_{k-1}^{n_k})=n_k\overline{b_{k}}$ and
$\deg_y f_{k-1}^{n_k}=n_k\frac{n}{e_{k-1}}=\frac{n}{e_{k}}$. Therefore $n_k\overline{b_{k}}\leq \overline{b_{k+1}}$
by  Property (II$_k$).  By Lemma \ref{aritmetico} we get $n_k\overline{b_{k}}< \overline{b_{k+1}}$ since $\overline{b_{k+1}}\not\in  \mathbf N
\overline{b_0}+\cdots+\mathbf N \overline{b_{k}}$.
\end{proof}

\medskip

\noindent \begin{proof}[Proof of Theorem \ref{structure}]
The theorem follows by induction from Proposition \ref{maximal}, Proposition \ref{intermedio} and
from Remark \ref{n}.
\end{proof}

\begin{nota}
From Theorem \ref{structure} and Lemmas \ref{Ik}, \ref{IIk}, and \ref{IIIk} it follows that
Properties I$_k$, II$_k$ and III$_k$ hold for all $0\leq k<h$.
\end{nota}

\begin{nota}
\label{nota-key}
Let $\{f=0\}\neq \{x=0\}$ be a branch such that $n=i_0(f,x)>1$. Let 
$\overline{b_0},\ldots,\overline{b_h}$, $\overline{b_0}=n$ be the $n$-minimal system of generators of $\Gamma(f)$. Suppose that the first $k+1$ terms 
$\overline{b_0},\ldots,\overline{b_k}$, with $k<h$ are given. Let $f_k\in\bK[[x]][y]$ be a monic polynomial of degree $\frac{n}{e_k}$ such that $i_0(f,f_k)\not\equiv 0$ $(\hbox{\rm mod } e_k)$. Then by Property II$_k$ we get
$i_0(f,f_k)\leq \overline{b_{k+1}}$. The inequality $i_0(f,f_k)\geq \overline{b_{k+1}}$ follows from the fact that $i_0(f,f_k)\not\in 
\bN \overline{b_{0}}+\cdots+\bN \overline{b_{k}}$. We get $i_0(f,f_k)=\overline{b_{k+1}}$ and $f_k$ is a $k$-th key polynomial of $f$.
\end{nota}

\begin{ejemplos}

$\;$

\noindent Let $\bK$ be an algebraically closed field of characteristic $p>2$.
\begin{enumerate}
\item [A.] Let $f(x,y)=y^{p^2-1}+y^p-x^{p^2}=(y-x^p)^p+y^{p^2-1}$.  Put $(\phi(t),\psi(t))=(t^p+t^{p^2-1},t^{p^2})$.
We have $f(\phi(t),\psi(t))=0$ and $f$ is irreducible in $\bK[[x,y]]$. Since the order of $f$ is the prime number $p$ we get $\overline{\hbox{\rm char}} f=(\overline{b_0},\overline{b_1})$ where $\overline{b_0}=\ord f=p$. To calculate $\overline{b_1}$ we have to construct a monic polynomial $f_0=y+\cdots$ of degree $1$ such that $i_0(f,f_0)\not\equiv 0$ $($mod $e_0)$, $e_0=\overline{b_0}=p$. Take $g=y-x^p$. Then $i_0(f,g)=p(p^2-1)=i_0(f,x^{p^2-1})$. There is a constant $c\in \bK$ such that $i_0(f,g-cx^{p^2-1})>i_0(f,g)$. Using the parametrization $(\phi(t),\psi(t))$ we find that $c=-1$. Let $f_0=g-cx^{p^2-1}=y-x^p+x^{p^2-1}$.  Then $i_0(f,f_0)=p
^3+p^2-2p-1\not\equiv 0$ $($mod $p)$ and  by Remark \ref{nota-key} we get $\overline{b_1}=p
^3+p^2-2p-1$
and consequently $\overline{\hbox{\rm char}} f=(p,p
^3+p^2-2p-1)$.

\item [B.] Let $f(x,y)=y^{p^2-1}+y^{p^2-p}-x^{p^2}=(y^{p-1}-x^p)^p+y^{p^2-1}$ and $(\phi(t),\psi(t))=(t^{p^2-p}+t^{p^2-1},t^{p^2})$.
We have $f(\phi(t),\psi(t))=0$ and $f$ is irreducible in $\bK[[x,y]]$.
Since $\ord f=p^2-p$ and $i_0(f,y)=p^2\not\equiv 0$ $($ mod $(p^2-p))$ we get $ \overline{b_0}=p^2-p$, $\overline{b_1}=p^2$, $e_1=\gcd(\overline{b_0},\overline{b_1})=p$ and $\overline{\hbox{\rm char}} f=(\overline{b_0},\overline{b_1},\overline{b_2})$. To compute $\overline{b_2}$ we have to construct a monic polynomial $f_1\in \bK[[x]][y]$ of degree $\frac{\overline{b_0}}{e_1}=p-1$ such that $i_0(f,f_1)\not\equiv 0$ $($mod $p)$. Starting with the polynomial $g=y^{p-1}-x^p$ and proceeding like in Example A we find $f_1=y^{p-1}-x^p+x^{p+1}$. Since $i_0(f,f_1)=p^3-1$ we get $\overline{b_2}=p^3-1$ and consequently $\overline{\hbox{\rm char}} f=(p^ 2-p,p^2,p^3-1)$.
\end{enumerate}
\end{ejemplos}

\section{Key polynomials}
\label{key}

\noindent The key polynomials under the name of semi-roots were studied by Abhyankar \cite{Abhyankar2}
and Popescu-Pampu \cite{Popescu}. Here we propose the treatment without any restriction on the field
characteristic.

\medskip

\noindent Let $f=f(x,y)\in \bK[[x,y]]$ be an  irreducible power series such that
$i_0(f,x)=\ord f(0,y)=n<+\infty$ and let $\overline{\hbox{\rm char}}_x f=(\overline{b_0},\ldots,
\overline{b_h})$, $\overline{b_0}=n$. Let $k\in \{0,\ldots,h\}$. Recall that a monic polynomial $g\in \bK [[x]][y]$ is a
$k$-th {\em key polynomial} of $f$ if $\deg_y g=\frac{n}{e_k}$ and $v_f(g)=\overline{b_{k+1}}$.
By the Semigroup Theorem,  for any $k\in\{0,\ldots,h\}$ there exists a $k$-th key polynomial of $f$. We fix a sequence $f_0,\ldots, f_h$ of key polynomials of $f$ such that $f_k$ is a $k$-th key
polynomial.

\begin{prop}
Let $g\in \bK[[x]][y]$ be a $k$-th key polynomial of $f$. Then $g$ is an irreducible \hbox{\rm (}in $\bK[[x]][y]$\hbox{\rm )}
distinguished polynomial.
\end{prop}

\noindent \begin{proof}

\noindent Suppose that $g$
is not irreducible. Then $g=g_1g_2$ in $\bK[[x]][y]$ with
monic polynomials $g_1,g_2$ of positive degrees. Consequently,
$\deg_y g_1,\deg_y g_2<\deg_y g=\frac{n}{e_{k}}$ and by Property I$_k$
we get $i_0(f,g_1),i_0(f,g_2)\in \bN
\overline{b_0}+\cdots+\mathbf N \overline{b_{k}}$ and
$\overline{b_{k+1}}=i_0(f,g)=i_0(f,g_1)+i_0(f,g_2)\in \bN
\overline{b_0}+\cdots+\mathbf N \overline{b_{k}}$ which is a
contradiction. Therefore $g$ is irreducible in $\bK[[x]][y]$.
\medskip

\noindent To check that $g=g(x,y)\in \bK[[x]][y]$ is
distinguished observe that from irreducibility of $g$ in
$\bK[[x]][y]$ and from Hensel's Lemma we get
$g(0,y)=(y-c)^{n/e_{k}}$ in $\bK[y]$. On the other hand the
condition $i_0(f,g)=\overline{b_{k+1}}$ implies $g(0,0)=0$ since
$\overline{b_{k+1}}>0$. Hence $c=0$ and $g$ is a distinguished
polynomial. \end{proof}

\begin{coro}
\label{key polynomials} The key polynomials of $f\in \bK[[x,y]]$ are
distinguished and irreducible in $\bK[[x,y]]$.
\end{coro}

\noindent \begin{proof}
The corollary follows from the fact that a distinguished polynomial irreducible
in $\bK [[x]][y]$ is irreducible in $\bK [[x,y]]$ (see \cite{Abhyankar1}, p. 75).
\end{proof}

\begin{lema}
\label{il3}
Let $f=f(x,y)\in \bK[[x,y]]$ be an irreducible power series such that $n=i_0(f,x)<+\infty$ and let $(\overline{b_0},\ldots,\overline{b_h})$ be an $n$-characteristic sequence. Suppose that there exist monic polynomials $f_0, \ldots,f_{h-1}\in \bK[[x]][y]$ such that $\deg_y f_k=\frac{n}{e_k}$ and $i_0(f,f_k)=\overline{b_{k+1}}$ for $k\in\{0,\ldots,h-1\}$. Then 
$\overline{\hbox{\rm char}}_x f=(\overline{b_0},\ldots,
\overline{b_h})$ and $f_0, \ldots,f_{h-1}$ are key polynomials of $f$.
\end{lema}

\noindent \begin{proof} 
Recall that $\Gamma(f)=\{v_f(g)\;:\;g\in \bK[[x]][y]\backslash\{0\}\;:\;\deg_y g<n\}.$
By Lemma \ref{Ik} we get $\Gamma(f)=\bN \overline{b_0}+\cdots+\bN \overline{b_h}$. According to the first statement of Proposition \ref{prop-min} the sequence $\overline{b_0}, \ldots,\overline{b_h}$ is the $n$-minimal system of generators of the semigroup $\Gamma(f)$ and the lemma follows.
 \end{proof}

\begin{prop}
\label{p1}
Let $g$ be a $k$-th key polynomial of $f$. Then $g$ is a
distinguished polynomial, irreducible in
$\bK[[x,y]]$ with characteristic $\overline{\hbox{\rm char}}_x g
=\left(\frac{\overline{b_0}}{e_k},\ldots,\frac{\overline{b_k}}{e_k}\right)$.
Moreover the polynomials $f_0,f_1,\ldots,f_{k-1}$ are 
key polynomials of $g$.
\end{prop}

\noindent \begin{proof} We have already checked
that  the key polynomials are distinguished and
irreducible. Let us calculate
$i_0(f_i,g)$ for $i<k$. Consider $f_i,g,f$ and the log-distances
$d_x(f_i,g)=\frac{e_i e_k i_0(f_i,g)}{n^2}$,
$d_x(f_i,f)=\frac{e_i\overline{b_{i+1}}}{n^2}$ and
$d_x(g,f)=\frac{e_k\overline{b_{k+1}}}{n^2}$. The sequence
$\left(e_{i-1}\overline{b_i}\right)$ is strictly increasing, therefore
$d_x(f_i,f)<d_x(g,f)$ and by the STI we get
$d_x(f_i,g)=d_x(f_i,f)$ which implies
$i_0(f_i,g)=\frac{\overline{b_{i+1}}}{e_k}$.

\medskip

\noindent On the other hand $\deg_y f_i=\frac{n}{e_i}=
\frac{n}{e_k}:\frac{e_i}{e_k}$ and $\frac{e_i}{e_k}=\gcd\left(\frac{\overline{b_0}}{e_k},\ldots,
\frac{\overline{b_i}}{e_k}\right)$.
The proposition follows from Lemma \ref{il3}. \end{proof}

\medskip

\noindent We finish this section with

\begin{prop}
\label{adic-expansion}
Let $h\in\bK[[x]][y]$ be a $(k-1)$-th key polynomial of $f$ and
let $g\in \bK[[x]][y]$ be a monic polynomial such that $\deg_y g=\frac{n}{e_k}$
and $v_f(g)>n_k\overline{b_k}$. Let $g=h^{n_k}+a_1h^{n_k-1}+\cdots+a_{n_k}$,
$\deg_y a_i< \deg_y h=\frac{n}{e_{k-1}}$ be the $h$-adic expansion of $g$.
Then $v_f(a_i)>i\overline{b_k}$ if $1\leq i<n_k$ and $v_f(a_{n_k})=n_k\overline{b_k}$.

\end{prop}

\noindent \begin{proof}
Consider the $h$-adic expansion of $g$
\begin{equation}
\label{expansion}
g=h^{n_k}+a_1h^{n_k-1}+\cdots+a_{n_k},
\end{equation}

\noindent where $\deg_y a_i<\deg_y h=n/e_{k-1}$.

\noindent Let $I$ be the set of all $i\in\{1,\ldots,n_k\}$ such that $a_i\neq 0$.
Since $v_f(g)>n_k\overline{b_k}=v_f(h^{n_k})$, $I\neq \emptyset$.
There is $v_f(a_i)<+\infty$ for $i\in I$ and by Property I$_k$ we get
$v_f(a_i)\in \bN \overline{b_0}+\cdots + \bN \overline{b_{k-1}}$,  hence
$v_f(a_i)\equiv 0$ mod $e_{k-1}$ for every $i\in I$. We have

\begin{equation}
\label{noigual}
v_f(a_i h^{n_k-i})\neq v_f(a_j h^{n_k-j})
\end{equation}

\noindent for $i,j\in I$ with $i\neq j$.

\noindent Indeed, $v_f(a_i h^{n_k-i})=v_f(a_j h^{n_k-j})$ with $i<j$ implies, as in
the proof of Property I$_k$, the congruence $(j-i)\overline{b_k}/e_k\equiv 0$ mod
$n_k$, which leads to a contradiction for $0<j-i<n_k$.

\medskip

\noindent From (\ref{expansion}) and (\ref{noigual}) we have

\begin{equation}
\label{min}
v_f(g-h^{n_k})=\min_{i=1}^{n_k}v_f(a_ih^{n_k-i}).
\end{equation}

\noindent By assumption $v_f(g)>n_k\overline{b_k}=v_f(h^{n_k})$, so $v_f(g-h^{n_k})=n_k\overline{b_k}$
and (\ref{min}) implies $n_k\overline{b_k}\leq v_f(a_ih^{n_k-i})=v_f(a_i)+(n_k-i)\overline{b_k}$ for
$i\in \{1,\ldots,n_k\}$. Therefore we get

\begin{equation}
\label{ineq}
v_f(a_i)\geq i\overline{b_k}
\end{equation}

\noindent for $i\in \{1,\ldots,n_k\}$.

\noindent Moreover,  
\begin{equation}
\label{e1}
\hbox{\rm if}\; v_f(a_i)=i\overline{b_k}\; \hbox{\rm for } i\in \{1,\ldots,n_k\}
\;\hbox{\rm then } i=n_k.
\end{equation}

\noindent Indeed, from $v_f(a_i)=i\overline{b_k}$ it follows that $i\overline{b_k}\equiv 0$
mod $e_{k-1}$ and $i\overline{b_k}/e_k\equiv 0$ mod $n_k$, so $i\equiv 0$ mod $n_k$ because
$\overline{b_k}/e_k$ and $n_k$ are coprime. Hence we get $i=n_k$. According to (\ref{min}) there
exists $i_0\in I$ such that $v_f(a_{i_0}h^{n_k-i_0})=v_f(g-h^{n_k})=n_k\overline{b_k}$. Thus
$v_f(a_{i_0})=i_0\overline{b_k}$ and by (\ref{e1}) we get $i_0=n_k$.
\end{proof}

\medskip
\noindent {\bf Notes}

\medskip
\noindent Key polynomials of $f$  introduced in \cite{MacLane}, define curves of maximal contact with \{$f=0$\} (see \cite{LJ}) and are connected
with {\em curvettes} associated with extremal points in the dual graph of \{$f=0$\} (see, for example \cite{GB} p. 54, \cite{Popescu} p.13). They also play an important role in studying  valuations
\cite{Spivakovsky}.

\section{The Abhyankar-Moh theory}

\label{Abhyankar-Moh theory}
\noindent We are going to prove the Abhyankar-Moh Theorem  on approximate roots using
the properties of key polynomials explained in Section \ref{key}. First let us recall the basic notions of Abhyankar-Moh theory (see \cite{Abhyankar-Moh}, \cite{Abhyankar3} or \cite{Popescu}).

\medskip

\noindent Let $R$ be an integral domain and let $d>1$ be a positive integer such that $d$ is a unit in
$R$. Denote $\deg f:=\deg_y f$ the degree of the polynomial $f\in R[y]$ 
in one variable $y$ and
assume that $d$ divides $\deg f$. According to Abhyankar and Moh (\cite{Abhyankar-Moh}, Section 1)
the {\em approximate d-th} root of $f$, denoted by $\sqrt[d]{f}$ is defined to be the unique
monic polynomial satisfying $\deg (f-\left(\sqrt[d]{f}\right)^d)<\deg f- \deg \sqrt[d]{f}$. For
the existence and uniqueness of $\sqrt[d]{f}$ see
\cite{Abhyankar-Moh}. We put by convention $\sqrt[1]{f}=f$. Obviously $\deg \sqrt[d]{f}=\frac{\deg f}{d}$. From the definition it follows that
$\sqrt[e]{\sqrt[d]{f}}=\sqrt[ed]{f}$ if $ed$ is a unit which divides $\deg f$ (see \cite{G-P2}).

\medskip

\noindent Given any monic polynomial $g\in R[y]$ of degree $\deg f/d$ we have the $g$-adic expansion of $f$, namely
\[f=g^d+a_1g^{d-1}+\cdots+a_d,\]

\noindent where $a_i\in R[y]$, $\deg a_i<\deg g$.

\medskip

\noindent The polynomials $a_i$ are uniquely determined by $f$ and $g$.

\medskip

\noindent The {\em Tschirnhausen operator} $\tau_f(g):=g+\frac{1}{d}a_1$ maps $g$ to
$\tau_f(g)$ which is again a monic polynomial of degree $\deg f/d$. One checks (see \cite{Abhyankar-Moh}, Section 1
and Section 6) that

\begin{enumerate}
\item $a_1=0$ if and only if $g=\sqrt[d]{f}$,
\item if $f=(\tau_f(g))^d+\overline{a_1}(\tau_f(g))^{d-1}+\cdots+\overline{a_d}$ is the
$\tau_f(g)$-expansion of $f$ then $\deg \overline{a_1}<\deg a_1$ or $\overline{a_1}=0$.

\noindent Using the above properties we get

\item $\sqrt[d]{f}=\tau_f(\tau_f \cdots (\tau_f(g)))$ with $\tau_f$ repeated $\deg f/d$ times.
\end{enumerate}

\medskip

\noindent Let $f=f(x,y)\in \bK[[x]][y]$ be an irreducible distinguished polynomial of degree $n>1$ such that $\overline{\hbox{\rm char}}_x f=(\overline{b_0},\ldots,\overline{b_h})$, $\overline{b_0}=n$.

\begin{prop}
\label{pre-A}
Let $g=g(x,y)\in \bK[[x]][y]$ be a monic polynomial such that
$\deg_y g=\frac{n}{e_k}$ and $v_f(g)>n_k\overline{b_k}$. Assume that
$n_k\not\equiv 0$ mod char $\bK$. Then
\begin{enumerate}
\item [(i)] if $h$ is a $(k-1)$-th key polynomial of $f$ then $\tau_g(h)$ is a $(k-1)$-th
key polynomial of $f$ as well,

\item [(ii)] $v_f(\sqrt[n_k]{g})=\overline{b_k}$.
\end{enumerate}
\end{prop}

\noindent \begin{proof}
Consider the $h$-adic development of $g$: $g=h^{n_k}+a_1h^{n_k-1}+\cdots+a_{n_k}$. By
Proposition \ref{adic-expansion} we get $v_f(a_1)>\overline{b_k}$ (because
$n_k>1$). Therefore $v_f(\tau_g(h))=v_f(h+\frac{1}{n_k}a_1)=v_f(h)=\overline{b_k}$.
Clearly $\deg_y \tau_g(h)=\deg_ y h$ and $(i)$ follows.

\medskip

\noindent To check $(ii)$ use $\deg_y g/n_k=n/e_{k-1}$ times $(i)$ and the
formula for the approximate root $\sqrt[n_k]{g}$ in terms of $\tau_g$.
\end{proof}
\medskip

\noindent Now we can prove the Abhyankar-Moh Theorem (see 
\cite{Abhyankar-Moh}).

\begin{teorema}[Abhyankar-Moh Fundamental Theorem on approximate roots]
\label{Abhyankar-Moh}
Let $f=f(x,y)\in \bK[[x]][y]$ be an irreducible distinguished polynomial of degree
$n>1$ with $\overline{\hbox{\rm char}}_x f=(\overline{b_0},\overline{b_1},\ldots,
\overline{b_h})$ and $\overline{b_0}=v_f(x)=n$. Let $1\leq k \leq h+1$. Suppose that
$e_{k-1}\not\equiv 0$ mod $char \bK$. Then:

\begin{enumerate}
\item $v_f(\sqrt[e_{k-1}]{f})=\overline{b_k}$,
\item $\sqrt[e_{k-1}]{f}$ is an irreducible distinguished polynomial of degree $n/e_{k-1}$ such
that $\overline{\hbox{\rm char}}_x \sqrt[e_{k-1}]f=(\overline{b_0}/e_{k-1},\overline{b_1}/e_{k-1},\ldots,
\overline{b_{k-1}}/e_{k-1})$.
\end{enumerate}
\end{teorema}

\noindent \begin{proof}
According to Proposition \ref{p1} it suffices to check the first part of the theorem.
We use descendent induction on $k$. If $k=h+1$ then $e_{k-1}=e_h=1$, $\overline{b_k}=
\overline{b_{h+1}}=+\infty$ and obviously $v_f(\sqrt[e_h]{f})=\overline{b_{h+1}}$. Let
$k\leq h$. Suppose that $e_k\not\equiv 0$ (mod char $\bK$) and $v_f(\sqrt[e_k]{f})=
\overline{b_{k+1}}$. The polynomial $\sqrt[e_k]{f}$ is of degree $n/e_k$ and $v_f
(\sqrt[e_k]{f})>n_k\overline{b_{k}}$ so we can apply Proposition \ref{pre-A} (ii) to
$g=\sqrt[e_k]{f}$ to get $v_f(\sqrt[n_k]{g})=\overline{b_k}$ provided that $n_k\not\equiv 0$
(mod char $\bK$).

\medskip

\noindent Assume that $e_{k-1}\not\equiv 0$ (mod char $\bK$). Then $e_k,n_k\not\equiv 0$
(mod char $\bK$) and we have $\sqrt[n_k]{g}=\sqrt[n_k]{\sqrt[e_k]{f}}=\sqrt[e_{k-1}]{f}$.
Consequently, $v_f(\sqrt[e_{k-1}]{f})=\overline{b_k}$ and we are done.
\end{proof}

\begin{coro}
Suppose that $n\not\equiv 0$ (mod char $\bK$). Then $\sqrt[e_{0}]{f}$, $\sqrt[e_{1}]{f},\ldots,$
$\sqrt[e_{h}]{f}$ is a sequence of key polynomials of $f$.
\end{coro}

\medskip

\noindent We say that a projective plane curve $C$ defined over $\bK$ {\em has one
branch at infinity} if there is a line $L$ (line at infinity) intersecting $C$ in only one 
point $O$, and $C$ is analytically irreducible at $O$, that is the equation of $C$ in an affine
coordinate system centered at $O$ is irreducible in the ring of formal power series.

\medskip

\noindent In what follows we denote by $n$ the degree of $C$, by $n'$ the multiplicity of $C$ at
$O$ and we put $d:=\gcd(n,n')$.

\medskip

\noindent We call $C$ {\em permissible} if $d\not\equiv 0$ (mod char $\bK$).

\medskip

\noindent Let us denote by $\Gamma_O$ the semigroup of the branch at infinity of the curve
$C$. Since $C$ and $L$ intersect with multiplicity $n$ at $O$ there exists the $n$-minimal system
of generators of $\Gamma_O$.

\begin{teorema}[Abhyankar-Moh inequality]
\label{Abhyankar-Moh inequality}
Assume that $C$ is a curve of degree $n>1$ with one branch at infinity and let 
$(\overline{b_0},\ldots,\overline{b_h})$ be the $n$-minimal system of generators of the 
semigroup $\Gamma_O$. If $C$ is permissible, then $e_{h-1}\overline{b_h}<n^2$.
\end{teorema}

\noindent \begin{proof}
Let $(x,y)$ be an affine coordinate system centered at $O$ such that $L$ has the
equation $x=0$. Let $f(x,y)=0$ be a polynomial equation of $C$ of total degree $n$.
Multiplying $f$ by a constant we may assume that $f$ is an $y$-distinguished polynomial (of degree
$n$ since $C$ and $L$ intersect only at $O$) irreducible in $\bK[[x,y]]$.  Let $\overline{\hbox{\rm
char}_x}f=(\overline{b_0},\ldots,\overline{b_h})$. We have $\overline{b_0}=n$ and $\overline{b_1}=n'$. 
Therefore $e_1=\gcd(\overline{b_0},\overline{b_1})=\gcd(n,n')\not\equiv 0$ (mod char $\bK$) and consequently
$e_{h-1}\not\equiv 0$ (mod char $\bK$).

\medskip

\noindent By Theorem \ref{Abhyankar-Moh} applied in the case $k=h$ the approximate root $\sqrt[e_{h-1}]{f}$
exists and $i_0(f,\sqrt[e_{h-1}]{f})=\overline{b_h}$. One checks that the total degree of $\sqrt[e_{h-1}]{f}$
is $\frac{n}{e_{h-1}}$ (see, for example \cite{G-P2}, Remark on p. 201). Thus by B\'ezout's theorem applied to $f$ and
$\sqrt[e_{h-1}]{f}$ we get $\overline{b_h}=i_0(f,\sqrt[e_{h-1}]{f})\leq n\frac{n}{e_{h-1}}$. In fact, we have
$\overline{b_h}< n\frac{n}{e_{h-1}}$ for $\overline{b_h}= n\frac{n}{e_{h-1}}$ would imply $\overline{b_h}\equiv 0$
(mod $e_{h-1}$) which is impossible.
\end{proof}

\medskip

\noindent The Abhyankar-Moh inequality implies an upper bound for the conductor $c$ of the semigroup $\Gamma_O$:

\begin{coro}
\label{coro-ABi}
Under the notations and assumptions of Theorem \ref{Abhyankar-Moh inequality},
$$c\leq (n-1)^2-(n_1-1)(n-\overline{b_1})=(n-1)^2-\left(\frac{n}{d}-1\right)(n-n').$$
\end{coro}

\noindent \begin{proof}
By the Abhyankar-Moh inequality we have $e_{k-1}\overline{b_k}<n^2$ for $k\in\{1,\ldots,h\}$ since the sequence
$(e_{k-1}\overline{b_k})$ is increasing. By the Conductor formula (Corollary \ref{ST2}) we get
\begin{eqnarray*}
c&=&\sum_{k=1}^h(n_k-1)\overline{b_k}-\overline{b_0}+1=(n-1)^2-\sum_{k=1}^h(n_k-1)\left(\frac{n^2}{e_{k-1}}-\overline{b_k}\right)\\
& = & (n-1)^2-(n_1-1)(n-\overline{b_1})-\sum_{k=2}^h(n_k-1)\left(\frac{n^2}{e_{k-1}}-\overline{b_k}\right)\\
&\leq&(n-1)^2-(n_1-1)(n-\overline{b_1})=(n-1)^2-\left(\frac{n}{d}-1\right)(n-n').
\end{eqnarray*}

\noindent The last equality follows from the fact that $\overline{b_1}=\min(\Gamma_O\backslash\{0\})
=\hbox{\rm mult}_O C=n'$.
\end{proof}

\medskip

\noindent Now we can prove
\begin{teorema}[Abhyankar-Moh Embedding Line Theorem, first formulation]
\label{AMel1}
Assume that $C$ is a rational projective irreducible curve of degree $n>1$ with one branch at infinity 
and such that the center of the branch at infinity $O$ is the unique singular point of $C$. Suppose that $C$ 
is permissible and let $n'$ be the multiplicity of $C$ at $O$. Then $n-n'$ divides $n$.
\end{teorema}

\noindent \begin{proof}
Let $c$ be the conductor of the semigroup $\Gamma_O$. Then we have

\begin{equation}
\label{ineq-coro}
c\leq (n-1)^2-\left(\frac{n}{d}-1\right)(n-n')
\end{equation}

\noindent by Corollary \ref{coro-ABi}.
\medskip

\noindent On the other hand from the Noether formula for the genus of projective plane curve we get
\begin{equation}
\label{eq-coro}
c=(n-1)(n-2).
\end{equation}

\noindent Combining (\ref{ineq-coro}) and (\ref{eq-coro}) we obtain
$\left(\frac{n}{d}-1\right)(n-n')\leq n-1<n$ and $\left(\frac{n}{d}-1\right)\left(\frac{n}{d}-\frac{n'}{d}\right)<\frac{n}{d}$.
Therefore we get $\frac{n}{d}-\frac{n'}{d}=1$ and $n-n'$ divides $n$.
\end{proof}
\medskip

\begin{nota}
From  the proof of Theorem \ref{AMel1} we also conclude that $\gcd(n,n')=n-n'$.
\end{nota}

\medskip

\begin{nota}
Let us keep the assumptions of Theorem \ref{AMel1} and let $\overline{\beta_0},\overline{\beta_1},\ldots$ be the minimal  sequence of generators of the semigroup $\Gamma_O$. Then $n'=\overline{\beta_0}$ and since $n$ is the intersection multiplicity of the branch at infinity with a nonsingular branch we have $n\leq \overline{\beta_1}$ and $n\equiv 0$ (mod $\overline{\beta_0})$ if $n\neq \overline{\beta_1}$. We claim that if $n\neq \overline{\beta_1}$ then $n=2\overline{\beta_0}$. Indeed, if $n\neq \overline{\beta_1}$ then $n=a\overline{\beta_0}$ for an integer $a>0$ and $n=b(n-\overline{\beta_0})$ for an integer $b>0$ by Theorem \ref{AMel1}. Thus we get $a=(a-1)b$ which implies $a=2$. 

\medskip

\noindent We say that a nonsingular projective curve $C'$ has maximal contact with $C$ at infinity if $C'$ intersects $C$ at $0$ with multiplicity
$\overline{\beta_1}$ (see \cite{LJ}). The line at infinity has maximal contact with $C$ if and only if  $n\neq 2n'$. If $n = 2n'$ then there exists a  
nonsingular curve $C'$ of degree $2$ which has maximal contact with $C$ at infinity (if $f=0$ is the affine equation of $C$ such that in the proof of
Theorem \ref{Abhyankar-Moh inequality} then $C'$ is given by the equation $\sqrt[n']{f}=0$).
\end{nota}

\medskip

\noindent A polynomial mapping $F=(P,Q):\bK\longrightarrow \bK^2$ is called a polynomial embedding (of the line $\bK$) 
if there is a polynomial $G:\bK^2\longrightarrow \bK$ such that $G\circ F$ is equal to the identity of $\bK$.

\begin{teorema}[Abhyankar-Moh Embedding Line Theorem, second formulation]
If $F=(P,Q):\bK\longrightarrow \bK^2$ is a polynomial embedding such that $m=\deg P$, $n=\deg Q>0$ and $\gcd(m,n)
\not\equiv 0$ (mod char $\bK$) then $m$ divides $n$ or $n$ divides $m$.
\end{teorema}

\noindent \begin{proof}
We may assume that $1<m<n$. Let $C$ be the projective closure of the affine curve $F(\bK^2)$. Then $C$ is 
irreducible, rational with one branch at infinity centered at $O=(0:1:0)$. Moreover $C$ is of degree $n=\deg Q$ and 
its multiplicity at $O$ is $n'=\deg Q-\deg P=n-m$. Therefore $C$ is permissible. Apply Theorem \ref{AMel1} to the curve $C$.
\end{proof}

\medskip
\noindent {\bf Notes}

\medskip

\noindent S.S. Abhyankar and T.T. Moh developed the theory of approximate roots of polynomials with coefficients in the meromorphic series field $\bK((x))$ in the fundamental paper \cite{Abhyankar-Moh}. In 
\cite{Abhyankar-Moh2} they applied approximate roots to prove the Embedding Line Theorem. Later on Abhyankar in \cite{Abhyankar3} gave a simplified version of \cite{Abhyankar-Moh} and \cite{Abhyankar-Moh2}. The approach of Abhyankar and Moh is based on the technique of deformations of power series. H. Pinkham in \cite{Pinkham} 
proposed a method of eliminating the deformations which works in the {\em algebroid case} $\bK[[x]][y]$. P. Russel in \cite{Russell} used the Hamburger-Noether expansions to reprove the Abhyankar-Moh results (in the algebroid case) with weaker assumptions on the field characteristic. In our presentation of the subject we followed \cite{G-P2} (see also\cite{Chang-Wang}, \cite{Chang}).  The reader will find in \cite{Popescu} more references on the approximate roots. The assumption $C$ {\em is permissible} in the Embedding Line Theorem is relevant (see \cite{Nagata}).

\section{A formula for the intersection multiplicity of two branches}

\label{intersection multiplicity}
\noindent The aim of this section is to give a new formula for the intersection multiplicity of two branches.

\medskip

\noindent Let $\{f=0\}$ and $\{g=0\}$ be two branches different from
$\{x=0\}$. Let $\overline{\hbox{\rm
char}_x}f=(\overline{b_0},\ldots,\overline{b_h})$,
$\overline{b_0}=n=i_0(f,x)$ and $\overline{\hbox{\rm
char}_x}g=(\overline{b'_0},\ldots,\overline{b'_{h'}})$,
$\overline{b'_0}=n'=i_0(g,x)$. We denote by $f_0,\ldots,f_h$ and
$g_0,\ldots,g_{h'}$ key polynomials of $f$ and $g$, respectively.

\medskip

\begin{lema}
\label{igualdades} The equalities $\frac{\overline{b_i}}{n}=
\frac{\overline{b'_i}}{n'}$ for all $i\in \{1,\ldots,k\}$ imply
$\frac{n}{e_i}=\frac{n'}{e'_i}$ and $\frac{\overline{b_i}}{e_i}=
\frac{\overline{b'_i}}{e'_i}$ for all $i\in \{1,\ldots,k\}$.
\end{lema}

\noindent \begin{proof} We get $ne'_i=n\gcd (\overline{b'_0},\ldots,\overline{b'_i})=\gcd
(n\overline{b'_0},\ldots,n\overline{b'_i})=n'e_i$. Thus
$\frac{n}{e_i}=\frac{n'}{e'_i}$ and consequently
$\frac{\overline{b_i}}{e_i}= \frac{\overline{b'_i}}{e'_i}$ for all
$i\in \{1,\ldots,k\}$ since $\frac{\overline{b_i}}{n}=
\frac{\overline{b'_i}}{n'}$. \end{proof}

\medskip

\begin{teorema}
\label{igualdad polinomios} Let $n=i_0(f,x)>1$ and suppose that
$\frac{i_0(f,g)}{i_0(x,g)}> \frac{e_{k-1}\overline{b_k}}{n}$ for an
integer $k\in \{1,\ldots,h\}$. Then $k\leq h'$ and
$\frac{\overline{b_i}}{n}= \frac{\overline{b'_i}}{n'}$ for all $i\in
\{1,\ldots,k\}$. The first $k$ key polynomials $f_0,\ldots,f_{k-1}$
of $f$ are the first $k$ key polynomials of $g$.
\end{teorema}

\noindent \begin{proof}

\noindent Let us start with

\begin{propiedad}
\label{propiedad1} $ni_0(g,f_{i-1})=n'\overline{b_i}$ for $i\in
\{1,\ldots,k\}$.
\end{propiedad}

\noindent \begin{proof} Fix $i\in \{1,\ldots,k\}$ and consider the
power series $f,f_{i-1}$ and $g$. We have  $d_x(f,f_{i-1})=\frac{e_{i-1}\overline{b_i}}{n^2}$,
$d_x(f,g)=\frac{i_0(f,g)}{nn'}>\frac{e_{k-1}\overline{b_k}}{n^2}$
(by  assumption) and $d_x(g,f_{i-1})=\frac{e_{i-1}i_0(g,f_{i-1})}{nn'}$. Since $d_x(f,f_{i-1})<d_x(f,g)$ by  the STI we get
$d_x(g,f_{i-1})=d_x(f,f_{i-1})$, which implies the property.
\end{proof}

\begin{propiedad}
\label{propiedad2} $n'\equiv 0$  $($mod $\frac{n}{e_k})$.
\end{propiedad}

\noindent \noindent \begin{proof} We may write $e_k=a_0\overline{b_0}+a_1
\overline{b_1}+\cdots+a_k\overline{b_k}$ with $a_0,\ldots,a_k\in
\mathbf Z$ since $e_k=\gcd (\overline{b_0},\ldots,\overline{b_k})$. Hence we get
$e_kn'=(a_0n')n+a_1(n'\overline{b_1})+\cdots+a_k(n'\overline{b_k})\equiv
0$ (mod $n$) by Property \ref{propiedad1} and consequently $n'\equiv
0$ (mod $\frac{n}{e_k}$).
\end{proof}

\begin{propiedad}
\label{propiedad3} Let $i>0$ be an integer. Then
$d_x(g,f_{i-1})=\frac{e_{i-1}\overline{b_i}}{n^2}$ for $i\leq k$,
$d_x(g,g_{i-1})=\frac{e'_{i-1}\overline{b'_i}}{(n')^2}$ for $i\leq
h'$, and
$d_x(f_{i-1},g_{i-1})=\frac{e_{i-1}e'_{i-1}i_0(f_{i-1},g_{i-1})}{nn'}$
for $i\leq \hbox{\rm min}(k,h')$.
\end{propiedad}

\noindent \noindent \begin{proof} We have $d_x(g,f_{i-1})=\frac{e_{i-1}
i_0(g,f_{i-1})}{n'n}=\frac{e_{i-1}\overline{b_i}}{n^2}$ by Property
\ref{propiedad1}. The formulae for $d_x(g,g_{i-1})$ and
$d_x(f_{i-1},g_{i-1})$ follow from the definitions.
\end{proof}

\begin{propiedad}
\label{propiedad-4}
We have $h'\geq 1$ and
$\frac{\overline{b_1}}{n}=\frac{\overline{b'_1}}{n'}$.
\end{propiedad}

\noindent \noindent \begin{proof} \label{propiedad4} From Property \ref{propiedad2} it
follows $n'>1$ since $\frac{n}{e_k}>1$ for $k>0$. Thus $h'\geq 1$
and we may apply Property \ref{propiedad3} for $i=1$. We get
$d_x(g,f_0)=\frac{\overline{b_1}}{n}\not\in \mathbf N$,
$d_x(g,g_0)=\frac{\overline{b'_1}}{n'}\not\in \mathbf N$ and
$d_x(f_0,g_0)=i_0(f_0,g_0)\in \mathbf N$. By the STI we obtain
$\frac{\overline{b_1}}{n}=\frac{\overline{b'_1}}{n'}$.
\end{proof}

\begin{propiedad}
\label{propiedad5} Let $i>0$ be an integer such that $i<k$, $i\leq
h'$ and $\frac{\overline{b_j}}{n}=\frac{\overline{b'_j}}{n'}$ for
all $j\leq i$. Then $i<h'$ and
$\frac{\overline{b_{i+1}}}{n}=\frac{\overline{b'_{i+1}}}{n'}$.
\end{propiedad}

\noindent \noindent \begin{proof} From the assumption
$\frac{\overline{b_j}}{n}=\frac{\overline{b'_j}}{n'}$ for all $j\leq
i$ and from Lemma \ref{igualdades} we get
$\frac{e_i}{n}=\frac{e'_i}{n'}$. By Property \ref{propiedad2} we may
write $n'=l\frac{n}{e_k}$,  where $l>0$ is an integer. Thus
$e'_i=n'\frac{e_i}{n}=l\frac{e_i}{e_k}>1$ since $i<k$. From $e'_i>1$
we get obviously $i<h'$. Now we may apply Property \ref{propiedad3}
for the index $i+1$ since $i+1\leq k$ and $i+1\leq h'$. We get
$d_x(g,f_i)=\frac{e_i\overline{b_{i+1}}}{n^2}$,
$d_x(g,g_i)=\frac{e'_i\overline{b'_{i+1}}}{(n')^2}$ and
$d_x(f_i,g_i)=\left(\frac{e_i}{n}\right)\left(\frac{e'_i}{n'}
\right)i_0(f_i,g_i)$. Recall that $\frac{e_i}{n}=\frac{e'_i}{n'}$.
Note that $d_x(g,f_i)\neq d_x(f_i,g_i)$. Indeed if we had
$d_x(g,f_i)=d_x(f,g_i)$ then we would get
$\overline{b_{i+1}}=e_ii_0(f_i,g_i)$ which is impossible since
$\overline{b_{i+1}}\not\equiv 0$ (mod $e_i$). Similarly we check
that $d_x(g,g_i)\neq d_x(f_i,g_i)$. Using the STI we get
$d_x(g,f_i)=d_x(g,g_i)$, which implies
$\frac{\overline{b_{i+1}}}{n}=\frac{\overline{b'_{i+1}}}{n'}$.
\end{proof}
\medskip

\noindent  Now we can finish the proof of Theorem \ref{igualdad polinomios}

\medskip

\noindent From Properties \ref{propiedad-4} and  \ref{propiedad5} we
conclude that $k\leq h'$ and
$\frac{\overline{b_i}}{n}=\frac{\overline{b'_i}}{n'}$ for
$i\in\{1,\ldots,k\}$, which proves the first part of Theorem
\ref{igualdad polinomios}. Let $i\in\{0,1,\ldots,k-1\}$. By Property
\ref{propiedad1}
$i_0(g,f_{i-1})=\frac{n'\overline{b_i}}{n}=\overline{b'_i}$ since
$\frac{\overline{b_i}}{n}=\frac{\overline{b'_i}}{n'}$. Moreover
$\deg _y(f_{i-1})=\frac{n}{e_{i-1}}=\frac{n'}{e'_{i-1}}$ so
$f_{i-1}$ is a key-polynomial of $g$.
\end{proof}

\begin{nota}
Under the notations and assumptions of Theorem \ref{igualdad
polinomios}, we get

\[\frac{i_0(g,f)}{i_0(x,f)}=\frac{i_0(f,g)}{i_0(x,g)}
\frac{i_0(x,g)}{i_0(x,f)}>
\frac{e_{k-1}\overline{b_k}}{n}\frac{n'}{n}=\frac{e'_{k-1}}{n'}
\frac{\overline{b'_k}}{n'}n'=\frac{e'_{k-1}\overline{b'_k}}{n'}.\]
\end{nota}

\vspace{1cm}

\noindent Let $f,g\in \bK[[x,y]]$ be irreducible power series such
that the branches $\{f=0\}$ and $\{g=0\}$ are different from
$\{x=0\}$. We put, by definition:

$$k_x(f,g)=\hbox{\rm min}\left \{k>0\;:\;\frac{i_0(f,g)}{i_0(x,g)}\leq
\frac{e_{k-1}\overline{b_k}}{n}\right \}.$$

\medskip

\noindent Obviously, we have $1\leq k_x(f,g)\leq h+1$. Recall that
$\overline{b_{h+1}}=+\infty$ so $k_x(f,g)=h+1$ if and only if
$\frac{i_0(f,g)}{i_0(x,g)}>\frac{e_{h-1}\overline{b_h}}{n}.$

\medskip

\noindent On the other hand $k_x(f,g)=1$ means that
$\frac{i_0(f,g)}{i_0(x,g)}\leq \overline{b_1}.$

\medskip

\noindent Note that $k_x(f,g)$  is different from the {\em coincidence exponent} defined by means of Puiseux' expansions in the case of charac\-teristic zero (see \cite{Popescu}, p. 299, \cite{G-P2}, p. 205).

\medskip

\begin{lema}
\label{calculo multiplicidad} Suppose that
$\frac{\overline{b_1}}{n}\neq \frac{\overline{b'_1}}{n'}$. Then
$i_0(f,g)\leq \hbox{\rm inf}\{n'\overline{b_1},n\overline{b'_1}\}$.
If $i_0(f,g)< \hbox{\rm inf}\{n'\overline{b_1},n\overline{b'_1}\}$
then $i_0(f,g)=nn'i_0(f_0,g_0)$.
\end{lema}

\noindent \begin{proof}
If $\frac{\overline{b_1}}{n}\neq \frac{\overline{b'_1}}{n'}$ then
$\overline{b_1}\neq +\infty$ or $\overline{b'_1}\neq +\infty$ that is
$n>1$ or $n'>1$. We may assume $n>1$. If $f$ and $g$ have a common
key-polynomial of degree $1$ then $i_0(f,g)=\hbox{\rm
inf}\{n'\overline{b_1},n\overline{b'_1}\}$. Indeed, if $f_0$ is a
key-polynomial of $f$ and $g$ then
$d_x(f,f_0)=\frac{\overline{b_1}}{n}$ and
$d_x(g,f_0)=\frac{\overline{b'_1}}{n'}$. From the assumption
$\frac{\overline{b_1}}{n}\neq \frac{\overline{b'_1}}{n'}$ we get
$d_x(f,g)=\hbox{\rm inf}\left\{\frac{\overline{b_1}}{n},
\frac{\overline{b'_1}}{n'}\right\}$ by the STI whence
$i_0(f,g)=\hbox{\rm inf}\{n\overline{b_1},n'\overline{b'_1}\}$.

\medskip

\noindent If $f$ and $g$ do not have a common key-polynomial of
degree $1$ then $i_0(f,g_0)<\overline{b_1}$ and
$i_0(g,f_0)<\overline{b'_1}$. Consequently we get
$d_x(f,g_0)<\frac{\overline{b_1}}{n}$,
$d_x(f,f_0)=\frac{\overline{b_1}}{n}$ and
$d_x(f_0,g_0)=d_x(f,g_0)<\frac{\overline{b_1}}{n}$ by the STI.
Analogously applying the STI to $g,f_0$ and $g_0$ we get
$d_x(f_0,g_0)=d_x(f_0,g)<\frac{\overline{b'_1}}{n'}$. We may assume
without loss of generality that $\inf\{n\overline{b_1'},n'\overline{b_1}\}=n\overline{b_1'}$. Then 
$\frac{\overline{b_1'}}{n'}<\frac{\overline{b_1}}{n}$ and $d_x(g,f_0)<d_x(f,f_0)$. By the STI we get $d_x(f,g)=d_x(g,f_0)$. Thus
 $d_x(f,g)=d_x(g,f_0)<\frac{\overline{b'_1}}{n'}$ and $i_0(f,g)<nn'\frac{\overline{b'_1}}{n'}=n\overline{b'_1}$. On the other hand $d_x(f,g)=d_x(g,f_0)=d_x(f_0,g_0)$, which gives $i_0(f,g)=nn'i_0(f_0,g_0)$.
\end{proof}

\begin{lema}
\label{cota} If $k_x(f,g)=k$ then $k\leq h'+1$ and
$\frac{i_0(f,g)}{i_0(f,x)}\leq \frac{e'_{k-1}\overline{b'_k}}{n'}$.
\end{lema}
\noindent \begin{proof}
Let $k=1$. Then $\frac{i_0(f,g)}{i_0(g,x)}\leq \overline{b_1}$. If
$\frac{\overline{b_1}}{n}=\frac{\overline{b'_1}}{n'}$ then
$$\frac{i_0(f,g)}{i_0(f,x)}=\frac{i_0(g,x)}{i_0(f,x)}\frac{i_0(f,g)}{i_0(g,x)}
\leq
\left(\frac{n'}{n}\right)\overline{b_1}=n'\left(\frac{\overline{b'_1}}{n'}\right)=\overline{b'_1}.
$$

\noindent If $\frac{\overline{b_1}}{n}\neq
\frac{\overline{b'_1}}{n'}$ then $i_0(f,g)\leq \hbox{\rm
inf}\{n'\overline{b_1},n\overline{b'_1}\}\leq n\overline{b'_1}$ by
Lemma \ref{calculo multiplicidad} and consequently
$\frac{i_0(f,g)}{i_0(x, f)}\leq \overline{b'_1}$.

\medskip

\noindent Now suppose that $k>1$. By definition of $k_x(f,g)$  of $f$ and $g$ we get
\begin{equation} \label{d1}
\frac{i_0(f,g)}{i_0(x,g)}> \frac{e_{k-2}\overline{b_{k-1}}}{n}.
\end{equation}

\noindent By Theorem \ref{igualdad polinomios} we get
$\frac{\overline{b_i}}{n}=\frac{\overline{b'_i}}{n'}$ for $i\leq
k-1\leq h'$. Moreover $f_0,\ldots,f_{k-2}$ are key polynomials of
$g$. If $e'_{k-1}=1$ then $h'=k-1$,
$\overline{b'_k}=\overline{b'_{h'+1}}=+\infty$ and the lemma is
obvious. Suppose that $e'_{k-1}>1$. Then $\overline{b'_k}<+\infty$.
If we suppose that Lemma \ref{cota} is not true then
\begin{equation} \label{d2}
\frac{i_0(f,g)}{i_0(x,f)}> \frac{e'_{k-1}\overline{b'_{k}}}{n'}.
\end{equation}

\noindent Applying Theorem \ref{igualdad polinomios} to $g$ and $f$
we get that $k-1<h'$, that is $k\leq h'$ and
$\frac{\overline{b_i}}{n}=\frac{\overline{b'_i}}{n'}$ for $i\leq k$.
So in particular
$\frac{\overline{b_k}}{n}=\frac{\overline{b'_k}}{n'}$ and by the
definition of $k$ we get $\frac{i_0(f,g)}{i_0(g,x)}\leq
\frac{e_{k-1}\overline{b_k}}{n}=e_{k-1}\left(\frac{\overline{b'_k}}{n'}\right)$
and $\frac{i_0(f,g)}{n}\leq 
\frac{n'}{n}e_{k-1}\frac{\overline{b'_k}}{n'}=n'\frac{e'_{k-1}}{n'}
\frac{\overline{b'_k}}{n'}=\frac{e'_{k-1}\overline{b'_k}}{n'}$ which
is a contradiction with inequality (\ref{d2}).
\end{proof}

\begin{teorema}[Formula for the intersection multiplicity]
\label{formula}
Let $f,g\in \bK[[x,y]]$ be irreducible power series such that
$n=i_0(f,x)<+\infty$ and $n'=i_0(g,x)<+\infty$. Let
$\overline{\hbox{\rm
char}_x}f=(\overline{b_0},\ldots,\overline{b_h})$,
$\overline{b_0}=n$ and $\overline{\hbox{\rm
char}_x}g=(\overline{b'_0},\ldots,\overline{b'_{h'}})$,
$\overline{b'_0}=n'$ be the Seidenberg characteristic sequences of the
branches $\{f=0\}$ and $\{g=0\}$. Let $k=k_x(f,g)$. Then we have
\begin{enumerate}
\item $\frac{\overline{b_i}}{n}=\frac{\overline{b'_i}}{n'},$ for
$i<k$,
\item $i_0(f,g)\leq \hbox{\rm inf}\{e'_{k-1}\overline{b_k},e_{k-1}\overline{b'_k}\}.$

\bigskip

\noindent Let $(f_0,\ldots,f_h)$ (resp. $(g_0,\ldots,g_{h'}))$ be a
sequence of key polynomials of $f$ (resp. of $g$).  We get

\bigskip

\item If $i_0(f,g)<\hbox{\rm inf}\{e'_{k-1}\overline{b_k},e_{k-1}
\overline{b'_k}\}$ then
$i_0(f,g)=e_{k-1}e'_{k-1}i_0(f_{k-1},g_{k-1})$,

\item if $k>1$ then
$e'_{k-2}\overline{b_{k-1}}=e_{k-2}\overline{b'_{k-1}}$ and
$i_0(f,g)>e'_{k-2}\overline{b_{k-1}}$.

\noindent Moreover $f_0,\ldots,f_{k-1}$ are the first $k$
key polynomials of $g$ and $g_0,\ldots,g_{k-1}$ are the first $k$
key polynomials of $f$.
\end{enumerate}
\end{teorema}

\noindent \begin{proof}
Part $(1)$ of the theorem  follows immediately  from  Theorem \ref{igualdad
polinomios}. From (1) we conclude, by Lemma \ref{igualdades}, that $\frac{n}{e_{k-1}}=\frac{n'}{e'_{k-1}}$.

\medskip

\noindent By the definition of $k_x(f,g)$ we get 
\[i_0(f,g)\leq i_0(x,g)\frac{e_{k-1}\overline{b_k}}{n}=n'\frac{e_{k-1}\overline{b_k}}{n}=e'_{k-1}\overline{b_k}.\]

\noindent On the other hand, by Lemma \ref{igualdades} we get
\[i_0(f,g)\leq i_0(x,f)\frac{e'_{k-1}\overline{b'_k}}{n'}=n\frac{e'_{k-1}\overline{b'_k}}{n'}=e_{k-1}\overline{b'_k}.\]

\noindent Combining the above inequalities we obtain Part (2) of the theorem.

\medskip

\noindent To check Part (3) suppose that $i_0(f,g)<\hbox{\rm
inf}\{e'_{k-1}\overline{b_k},e_{k-1}\overline{b'_k}\}.$ Then we have $\frac{i_0(f,g)}{i_0(x,g)}<
\frac{e_{k-1}\overline{b_k}}{n}$ and $\frac{i_0(f,g)}{i_0(x,f)}<
\frac{e'_{k-1}\overline{b'_k}}{n'}$$d_x(f,g)<\frac{e_{k-1}\overline{b_k}}{n^2}=d_x(f,f_{k-1})$ and
 $d_x(f,g)<\frac{e'_{k-1}\overline{b'_k}}{(n')^{2}}=d_x(g,g_{k-1})$. By
the STI applied to $f, f_{k-1}, g$ and to $f, g_{k-1}, g$ we get
$d_x(f,g)=d_x(f,g_{k-1})=d_x(g,f_{k-1})<d_x(f,f_{k-1}), d_x(g,g_{k-1})$. 
From $d_x(f,g)=d_x(g,f_{k-1})<d_x(g,g_{k-1})$ we conclude, by the STI applied to $f_{k-1}, g, g_{k-1}$ that $d_x(f,g)=d_x(f_{k-1}, g_{k-1})$, which implies $i_0(f,g)=e_{k-1}e'_{k-1} i_0(f_{k-1},g_{k-1})$. This proves Part (3).

\medskip

\noindent To check  Part $(4)$ suppose that $k>1$. Note that
$\frac{e'_{k-2}\overline{b_{k-1}}}{n'n}=\left(\frac{e'_{k-2}}{n'}\right)
\left(\frac{\overline{b_{k-1}}}{n}\right)=\left(\frac{e_{k-2}}{n}\right)
\left(\frac{\overline{b'_{k-1}}}{n'}\right)$ whence
$e'_{k-2}\overline{b_k}=e_{k-2}\overline{b'_k}$. Since $k>1$ we get,
by definition of $k_x(f,g)$, that
$\frac{i_0(f,g)}{i_0(x,g)}>\frac{e_{k-2}\overline{b_{k-1}}}{n}$, which
implies by the first part of the theorem  the inequality
$i_0(f,g)>e'_{k-2}\overline{b_{k-1}}$.

\medskip

\noindent The assertion on the key polynomials follows from the
second part of Theo\-rem \ref{igualdad polinomios}.
\end{proof}

\medskip

\begin{teorema}
\label{congruencia} Let $\{f=0\}$ and $\{g=0\}$ be two different
branches and let $\{l=0\}$ be a smooth branch. Suppose that
$n=i_0(f,l)<+\infty$, $n'=i_0(g,l)<+\infty$ and let $d=\gcd(n,n')$. Then $i_0(f,g)\equiv 0$ mod
$\left(\frac{n}{d}\;\hbox{\rm or }\frac{n'}{d}\right)$.
\end{teorema}

\noindent \begin{proof}
We may assume that $n,n'>1$ and $l=x$. Let $k=k_x(f,g)$. By Theorem
\ref{igualdad polinomios} and Lemma \ref{igualdades} we get
\begin{equation}
\label{i1} \frac{n}{e_i}=\frac{n'}{e'_i} \;\;\hbox{\rm for }i<k.
\end{equation}

\noindent On the other hand from the second and third part of Theorem \ref{formula}  it follows
\begin{equation}
\label{i2} i_0(f,g)\equiv 0 \;\;\hbox{\rm mod }(e_{k-1}\;\hbox{\rm
or }e'_{k-1}).
\end{equation}

\noindent From (\ref{i1}) we get 

\begin{equation}
\label{i3} e_{k-1}\equiv 0 \;\;\left(\hbox{\rm mod }\frac{n}{d}\right)\;\;and
\;\; e'_{k-1}\equiv 0 \;\;\left(\hbox{\rm mod }\frac{n'}{d}\right).
\end{equation}

\noindent Now (\ref{i2}) and (\ref{i3}) imply the theorem.
\end{proof}

\medskip

\noindent Using Theorem \ref{congruencia} we will prove the
following basic property of polynomial automorphisms of the plane (see \cite{Jung} and \cite{van der Kulk}).

\begin{teorema}[Jung-van der Kulk]
Let the mapping $(P,Q):\bK^2\longrightarrow \bK^2$ be a polynomial
automorphism. Then of the two integers $m=\deg P$,
$n=\deg Q$ one divides the other.
\end{teorema}

\noindent \begin{proof}(see \cite{van der Kulk})
Let $C$ and $D$ be projective curves with affine equations $P=0$ and
$Q=0$. Then $\deg D=n$, $\deg C=m$ and each of the
curves $C,D$ has exactly one branch at infinity.
 By B\'ezout's Theorem
these branches intersect with multiplicity $i=mn-1$. The line at
infinity cuts the branches of $C$ and $D$ with multiplicities $m$
and $n$ respectively. Thus by Theorem \ref{congruencia} we get
$i\equiv 0\;\;(\hbox{\rm mod }\frac{m}{d}\;\hbox{\rm or }
\frac{n}{d})$,  where $d=\gcd(m,n)$. This implies that
$m$ divides $n$ or $n$ divides $m$, since $i=mn-1$.
\end{proof}

\medskip
\noindent {\bf Notes}

\noindent The classical formula for the intersection multiplicity of two branches (see \cite{Hefez}, Chapter 8 or \cite{Popescu}) was well-known to geometers of the nineteenth century: H.J.S. Smith, G.H. Halphen and M. Noether. It was used by Zariski in \cite{Zariski-1971} to study the saturation of the local rings. Ancochea gave in \cite{Ancochea} a formula for the intersection multiplicity of two branches in terms of Hamburger-Noether expansions (see also \cite{Russell}, \cite{Campillo-libro}, \cite{Delgado2}).

\section{The Abhyankar-Moh irreducibility criterion}
\label{Ab-M}

\noindent Let $f\in \bK[[x,y]]$ be an irreducible power series such
that $\{f=0\}\neq \{x=0\}$. Let $n=i_0(f,x)>1$ and $\overline{\hbox
{\rm char}_x} f=(\overline{b_0},\ldots,\overline{b_h})$,
$\overline{b_0}=n$.

\begin{lema}
\label{igual caracteristica} Let $g=g(x,y)\in \bK[[x,y]]$ be an
irreducible power series such that $\{g=0\}\neq \{x=0\}$ and let $k$
be an integer such that $1\leq k\leq h$. If
$\frac{i_0(f,g)}{i_0(g,x)}>\frac{e_{k-1}\overline{b_k}}{n}$ then
$i_0(g,x)\equiv 0$  $\left(mod \;\frac{n}{e_k}\right)$. If, additionally,
$i_0(g,x)=\frac{n}{e_k}$ then $\overline{\hbox {\rm char}_x}
g=\left(\frac{\overline{b_0}}{e_k},\ldots,\frac{\overline{b_k}}{e_k}\right)$.
\end{lema}

\noindent \begin{proof}
It follows immediately from Theorem \ref{igualdad polinomios} and
Lemma \ref{igualdades}.
\end{proof}

\medskip

\begin{teorema}
\label{irreducibilidad} Let $g\in \bK[[x,y]]$ be a power series such
that $i_0(g,x)=\frac{n}{e_k}$ and $i_0(f,g)>n_k\overline{b_k}$ for
a $k\in\{1,\ldots,h\}$. Then $g$ is irreducible and
$\overline{\hbox {\rm char}_x}
g=\left(\frac{\overline{b_0}}{e_k},\ldots,\frac{\overline{b_k}}{e_k}\right)$.
\end{teorema}

\noindent \begin{proof}
Suppose that $i_0(f,g)>n_k\overline{b_k}$ and let $g=g_1\cdots g_s$
with irreducible $g_j\in \bK[[x,y]]$, for $j\in\{1,\ldots,s\}$. Then there
exists  $j\in \{1,\ldots,s\}$ such that
\begin{equation}
\label{ine}
\frac{i_0(f,g_j)}{i_0(g_j,x)}>\frac{e_{k-1}\overline{b_k}}{n}.
\end{equation}

\noindent Indeed, suppose that  inequality (\ref{ine}) is not true.
Then $i_0(f,g_j)\leq \frac{e_{k-1}\overline{b_k}}{n}i_0(g_j,x)$ for
all $j\in\{1,\ldots,s\}$ and we get $i_0(f,g)=\sum_{j=1}^s i_0(f,g_j)\leq
\sum_{j=1}^s
\frac{e_{k-1}\overline{b_k}}{n}i_0(g_j,x)=\frac{e_{k-1}\overline{b_k}}{n}
i_0(g,x)=n_k\overline{b_k}$ which contradicts the assumption about
$i_0(f,g)$. The inequality (\ref{ine}) implies by Lemma \ref{igual
caracteristica} that $i_0(g_j,x)=q\frac{n}{e_k}$ for some integer
$q>0$. On the other hand $i_0(g_j,x)\leq i_0(g,x)=\frac{n}{e_k}$.
Therefore $q=1$ and $i_0(g_j,x)=i_0(g,x)$. Recall that $g_j$ divides
$g$, $g_j$ is irreducible and $\ord g_j(0,y)=\ord g(0,y)$, thus $g_j$ is associated with $g$, which proves the
irreducibility of $g$. We get $\overline{\hbox {\rm char}_x}
g=\left(\frac{\overline{b_0}}{e_k},\ldots,\frac{\overline{b_k}}{e_k}\right)$
from the second part of Lemma \ref{igual caracteristica}.
\end{proof}

\begin{coro}[Abhyankar-Moh irreducibility criterion]
\label{coro-AM-irr}
If $i_0(g,x)=n$ and $i_0(f,g)>n_h\overline{b_h}$ then $g$ is irreducible and $\overline{\hbox{\rm char}_x} g=\overline{\hbox
{\rm char}_x} f$.
\end{coro}

\medskip

\noindent Using the Abhyankar-Moh inequality and the irreducibility criterion we prove

\begin{teorema} [Moh-Ephraim Pencil Theorem]
 \label{AMpencil}
Assume that $C$ is a plane curve of degree $n>1$ with one branch at infinity and that $C$ is permissible.
Let $D$ be another plane curve of degree $n>1$ with one point at infinity $O$ which is the unique common
point of $C$ and $D$. Then
\begin{enumerate}
\item the curve $D$ has one branch at infinity,
\item the branches at infinity of the curves $C$ and $D$ are equisingular.
\end{enumerate}
\end{teorema}

\noindent \begin{proof}
Let $(x,y)$ be an affine coordinate system centered at $O$ such that the line at infinity has the equation $x=0$. Let
$f(x,y)=0$ and $g(x,y)=0$ be the polynomial equations of the curves $C$ and $D$. Then 
$i_0(f,x)=i_0(g,x)=n$ and
$i_0(f,g)=n^2$ by
B\'ezout's theorem. Let $\overline{\hbox
{\rm char}_x} f=(\overline{b_0},\ldots,\overline{b_h})$. By the Abhyankar-Moh inequality we have
$e_{h-1}\overline{b_h}<i_0(f,g)$. Now, the assertions follow from Theorem \ref{irreducibilidad} (case $k=h$).
\end{proof}

\begin{nota}
Suppose that the plane curve $C$ of  the Moh-Ephraim pencil theorem is given by the homogeneous equation
$F(X,Y,Z)=0$ of degree $n$ and let $L:Z=0$ be the line at infinity. Consider the pencil $C_{\lambda}:F(X,Y,Z)-\lambda Z^n= 0$,
$\lambda \in \bK$. Applying Theorem \ref{AMpencil} to $C$ and $D=C_{\lambda}$, $\lambda\neq 0$ we prove that the
pencil $C_{\lambda}$ is equisingular at infinity.
\end{nota}

\medskip

\noindent {\bf Notes}

\noindent The Abhyankar-Moh irreducibility criterion was proved in \cite{Abhyankar-Moh} (Lemma 3.4) and explained in details in
\cite{Abhyankar3} (Theorem 12.4). The original version of the criterion was given for meromorphic curves. Using Puiseux series
the authors had to assume $n\not\equiv 0$ (mod char $\bK$). The version of the criterion presented in this paper is borrowed
from \cite{G-P2} where the result is proved for the case char $\bK=0$.

\medskip

\noindent The first part (irreducibility) of the Moh-Ephraim Pencil Theorem is due to Moh \cite{Moh2}, the second part
(equisingularity) to Ephraim \cite{Ephraim} (see also \cite{Chang}). In our treatment of the subject we do not need the
assumption char $\bK=0$, which is necessary in the quoted papers.

\section{Characterization of the semigroups associated with branches}
\label{carac-semi}
\noindent In this section we give a new proof of the well-known theorem on the existence of branches with given
semigroup (see \cite{Bresinsky} and \cite{Angermuller}). Following Teissier \cite{Teissier} we give explicitly the equation of a plane
curve with given characteristic. Our proof is written in the spirit of this paper, we do not use the technique of deformations.
Here is the main result of this section.

\begin{teorema}
\label{caract}
Let $\left(\overline{b_0},\ldots,\overline{b_h}\right)$ be an $n$-characteristic sequence. Suppose there exists a distinguished irreducible
polynomial $f_{h-1}\in \bK[[x]][y]$ such that  $\overline{\hbox{\rm char}_x} f_{h-1}=\left(\frac{\overline{b_0}}{e_{h-1}},\ldots,\frac{\overline{b_{h-1}}}{e_{h-1}}\right)$. Let $f_0,\ldots,f_{h-2}\in \bK[[x]][y]$ be a sequence of key polynomials of $f_{h-1}$. Let
$a_0,\ldots,a_{h-1}$ be the (unique) sequence of integers such that $a_0\overline{b_0}+a_1\overline{b_1}+\cdots+a_{h-1}
\overline{b_{h-1}}=n_h\overline{b_h}$,  where $0<a_0$ and
$0\leq a_i<n_i$ for $ i\in \{1,\ldots,h-1\}$ and let $c\in \bK \backslash \{0\}$. Put $f_h=f_{h-1}^{n_{h}}+cx^{a_0}f_0^{a_1}\cdots
f_{h-2}^{a_{h-1}}$. Then
\begin{enumerate}
\item $f_h$ is a distinguished irreducible polynomial of degree $n$, that is, $i_0(f_h,x)=\deg_y f_h$,
\item   $\overline{\hbox{\rm char}_x} f_h=\left(\overline{b_0},\ldots,\overline{b_h}\right)$ and $f_0,\ldots,f_{h-1}$ are key polynomials of $f_h$.
\end{enumerate}
\end{teorema}

\noindent \begin{proof}
Since $f_{h-1}$ is a distinguished polynomial of degree $\frac{n}{e_{h-1}}$ and $a_0>0$, we have

\begin{eqnarray*}
i_0(f_h,x)&=&i_0(f_{h-1}^{n_{h}}+cx^{a_0}f_0^{a_1}\cdots
f_{h-2}^{a_{h-1}},x)=i_0(f_{h-1}^{n_{h}},x)\\
&=&n_hi_0(f_{h-1},x)=n_h\frac{n}{e_{h-1}}=n.
\end{eqnarray*}

\noindent To calculate $\deg_y f_{h}$ observe that  $\deg_y f_{h-1}^{n_h}=n_h  \deg_y f_{h-1}=n_h\frac{n}{e_{h-1}}=n$
and $\deg_ycx^{a_0}f_0^{a_1}\cdots f_{h-2}^{a_{h-1}}=a_1\frac{n}{e_0}+\cdots+a_{h-1}\frac{n}{e_{h-2}}\leq
(n_1-1)\frac{n}{e_0}+\cdots+(n_{h-1}-1)\frac{n}{e_{h-2}}=\frac{n}{e_{h-1}}-1<n.$ Therefore we get $\deg_y f_h=n$. The proof that $f_h$ is irreducible is harder. We need auxiliary lemmas.

\begin{lema}
\label{l1}
$i_0(f_h,f_{h-1})=\overline{b_h}$.
\end{lema}

\noindent \begin{proof}

\begin{eqnarray*}
i_0(f_h,f_{h-1})&=&i_0(f_{h-1}^{n_{h}}+cx^{a_0}f_0^{a_1}\cdots
f_{h-2}^{a_{h-1}},f_{h-1})=i_0(x^{a_0}f_0^{a_1}\cdots
f_{h-2}^{a_{h-1}},f_{h-1})\\
&=& a_0i_0(x,f_{h-1})+a_1i_0(f_0,f_{h-1})+\cdots+a_{k-1}i_0(f_{h-2},f_{h-1})\\
&=&a_0\frac{\overline{b_0}}{e_{h-1}}+
a_1\frac{\overline{b_1}}{e_{h-1}}+\cdots+a_{h-1}\frac{\overline{b_{h-1}}}{e_{h-1}}=\frac{1}{e_{h-1}}n_h\overline{b_h}=
\overline{b_h}.
\end{eqnarray*}
\end{proof}

\begin{lema}
\label{l2}
There exists an irreducible factor $\phi$ of $f_h$ such that 
\[\frac{i_0(f_{h-1},\phi)}{i_0(\phi,x)}>\frac{n_{h-1}\overline{b_{h-1}}}{n}.\]
\end{lema}
\noindent \begin{proof}
Let $f_h=\phi_1\cdots \phi_s$ with irreducible factors $\phi_i\in \bK[[x,y]]$ for $i\in \{1,\ldots,s\}$. Suppose that
$\frac{i_0(f_{h-1},\phi)}{i_0(\phi,x)}\leq\frac{n_{h-1}\overline{b_{h-1}}}{n}$ for all $i\in\{1,\ldots,s\}$. By Lemma 
\ref{l1} we get
\begin{eqnarray*}
\overline{b_h}&=&i_0(f_h,f_{h-1})=\sum_{i=1}^s i_0(\phi_i,f_{h-1})\leq \sum_{i=1}^s \frac{n_{h-1}\overline{b_{h-1}}}{n}
i_0(\phi_i,x)\\
& =& \frac{n_{h-1}\overline{b_{h-1}}}{n}\sum_{i=1}^si_0(\phi_i,x)=\frac{n_{h-1}\overline{b_{h-1}}}{n}i_0(f_h,x)=n_{h-1}\overline{b_{h-1}}<\overline{b_h},
\end{eqnarray*}

\noindent which is a contradiction.
\end{proof}

\begin{lema}
\label{l3}
Let $\phi$ be an irreducible factor of $f_h$ such that $\frac{i_0(f_{h-1},\phi)}{i_0(\phi,x)}>\frac{n_{h-1}\overline{b_{h-1}}}{n}.$
Then there exists $\nu\in\{1,\ldots,n_h\}$ such that $i_0(\phi,x)=\nu\frac{\overline{b_0}}{e_{h-1}}$ 
and $i_0(\phi,f_k)=\nu \frac{\overline{b_{k+1}}}{e_{h-1}}$ for $k<h-1$. 
\end{lema}
\noindent \begin{proof}
Recall that  $\overline{\hbox{\rm char}_x} f_{h-1}=\left(\frac{\overline{b_0}}{e_{h-1}},\ldots,\frac{\overline{b_{h-1}}}{e_{h-1}}\right)$.
Applying Lemma \ref{igual caracteristica} to the irreducible power series $f_{h-1}$ and $\phi$ (note that $\frac{n_{h-1}\overline{b_{h-1}}}{n}=\frac{n_{h-1}\frac{\overline{b_{h-1}}}{e_{h-1}}}{\frac{n}{e_{h-1}}}$) we conclude that
$i_0(\phi,x)\equiv 0$ $\left(mod \; \frac{n}{e_{h-1}}\right)$. Therefore we can write $i_0(\phi,x)=\nu\frac{n}{e_{h-1}}$ with $\nu \leq e_{h-1}=n_h$ since $i_0(\phi,x)\leq i_0(f_h,x)=n$. Fix $k<h-1$ and consider the three branches $f_k=0$, $f_{h-1}=0$ and $\phi=0$. We get $d_x(f_k,\phi)=\frac{e_{h-1}e_ki_0(f_k,\phi)}{\nu n^2}$, $d_x(f_{h-1},\phi)=\frac{i_0(f_{h-1},\phi)}{i_0(\phi,x)\frac{n}{e_{h-1}}}>\frac{n_{h-1}\overline{b_{h-1}}}{n}\frac{e_{h-1}}{n}=\frac{e_{h-2}\overline{b_{h-1}}}{n^2},$
and
$d_x(f_{h-1},f_k)=\frac{\overline{b_{k+1}}/e_{h-1}}{(n/e_{h-1})(n/e_k)}
=\frac{e_k\overline{b_{k+1}}}{n^2}\leq \frac{e_{h-2}\overline{b_{h-1}}}{n^2}$,
 for $k<h-1$. 
Therefore $d_x(f_{h-1},f_k)<d_x(f_{h-1},\phi) $ and by the STI we get
$d_x(f_{h-1},f_k)=d_x(f_k,\phi)$, which implies $i_0(f_k,\phi)=\nu \frac{\overline{b_{k+1}}}{e_{h-1}}.$
\end{proof}

\medskip

\noindent Now we are in  a position to check that $f_h$ is an irreducible power series. Let $\phi$ be an irreducible factor of
$f_h$ such that in Lemma \ref{l2}. Since $f_h=f_{h-1}^{n_{h}}+cx^{a_0}f_0^{a_1}\cdots
f_{h-2}^{a_{h-1}}$ and $\phi$ is an irreducible factor of $f_h$ we get $i_0(f_{h-1}^{n_h},\phi)=i_0(x^{a_0}f_0^{a_1}\cdots
f_{h-2}^{a_{h-1}},\phi)$. Therefore, by Lemma \ref{l3} we have
\begin{eqnarray*}
n_hi_0(f_{h-1},\phi)&=&a_0i_0(x,\phi)+a_1i_0(f_0,\phi)+\cdots+a_{h-1}i_0(f_{h-2},\phi)\\
&=&a_0\nu\frac{\overline{b_0}}{e_{h-1}}+a_1\nu\frac{\overline{b_1}}{e_{h-1}}+\cdots+
a_{h-1}\nu\frac{\overline{b_{h-1}}}{e_{h-1}}=\frac{\nu}{e_{h-1}}n_h\overline{b_h}=\nu\overline{b_h}.
\end{eqnarray*}

\noindent Since $\nu\overline{b_h}\equiv 0$ (mod $n_h$) and $\overline{b_h}$, $n_h=e_{h-1}$ are coprime we
get $\nu\equiv 0$ (mod $n_h$) and $\nu=n_h$ because $1\leq \nu \leq n_h$.

\medskip

\noindent From Lemma \ref{l3} we get $i_0(\phi,x)=n_h\frac{\overline{b_0}}{e_{h-1}}=\overline{b_0}=n=i_0(f_h,x)$. Since $\phi$ divides $f_h$ we get
$f_h=\phi\psi$ in $\bK[[x,y]]$ with $\psi(0)\neq 0$. Therefore $f_h$ is irreducible.

\medskip

\noindent Now we prove the second statement of the theorem. First we check that $i_0(f_h,f_k)=\overline{b_{k+1}}$ for $k\in \{0,1,\ldots,h-1\}$.
We have $i_0(f_h,f_{h-1})=\overline{b_h}$ by Lemma \ref{l1}. Therefore
we may assume that $h>1$ and $k<h-1$. Applying Lemma \ref{l3} to the power series $\phi=f_h$ we get $i_0(f_h,f_k)=n_h\frac{\overline{b_{k+1}}}{e_{h-1}}=\overline{b_{k+1}}$ since $\nu=n_h$. Recall that $\deg_yf_k=\frac{n}{e_k}$ for $k\in\{0,1,\ldots,h-1\}$. Using Lemma 
\ref{il3} we conclude that $\overline{\hbox{\rm char}_x} f_h=\left(\overline{b_0},\ldots,\overline{b_h}\right)$ and that $f_0,\ldots,f_{h-1}$ is a sequence of key polynomials of $f_h$.
\end{proof}

\begin{teorema}[Bresinsky-Angerm\"uller]
\label{B-A}
Let $\overline{b_0},\ldots,\overline{b_h}$ be a sequence of positive integers. Then the following two conditions are equivalent:

\begin{enumerate}
\item There is an irreducible power series $f\in \bK[[x]][y]$ such that
$i_0(f,x)=\overline{b_0}$ and $\overline{b_0},\ldots,\overline{b_h}$ is the $\overline{b_0}$-minimal sequence of generators of the semigroup $\Gamma(f)$.
\item The numbers $\overline{b_0},\ldots,\overline{b_h}$ form a $\overline{b_0}$-characteristic sequence.
\end{enumerate}
\end{teorema}

\noindent \begin{proof}
The implication $(1)\Longrightarrow(2)$ follows from the Semigroup Theorem
(Theorem \ref{structure}). To check that  $(2)\Longrightarrow(1)$ we proceed by induction on the length $h$  of the characteristic sequence  using Theorem \ref{caract}. If $h=0$ then 
$(\overline{b_0})=(1)$ and we take $f=y$. Let $h>0$ and suppose that the implication $(2)\Longrightarrow(1)$ is true for $h-1$. Then there exists an irreducible distinguished polynomial $f_{h-1}\in\bK[[x]][y]$ such that $\Gamma(f_{h-1})=\frac{\overline{b_0}}{e_{h-1}}\bN+\cdots+
\frac{\overline{b_{h-1}}}{e_{h-1}}\bN$. Let $f_0,\ldots,f_{h-2}$ be a sequence of key polynomials of $f_{h-1}$. Take $f=f_{h-1}^{n_h}+x^{a_0}f_0^{a_1}\cdots f_{h-2}^{a_{h-1}}$,
where $0<a_0$ and $0\leq a_i<n_i$ for $i\in \{1,\ldots,h\}$ is  the (unique) sequence of integers such that $a_0\overline{b_0}+a_1\overline{b_1}+\cdots+a_{h-1}
\overline{b_{h-1}}=n_h\overline{b_h}$. Then by
Theorem \ref{caract} $f$ is an irreducible power series and $\Gamma(f)=\overline{b_0}\bN+\cdots+\overline{b_h}\bN$. 
\end{proof}

\medskip

\noindent Let $(\overline{b_0},\ldots,\overline{b_h})$  be an $n$-characteristic sequence. For any $k\in\{1,\ldots,h\}$ we have B\'ezout's relation $n_k\overline{b_k}=a_{k0}\overline{b_0}+a_{k1}\overline{b_1}+\cdots+
a_{kk-1}\overline{b_{k-1}}$,  where $a_{k0}>0$ and $0\leq a_{ki}<n_i$ for
$i\in \{1,\ldots,k-1\}$. Take $c_1,\ldots,c_h\in \bK\backslash \{0\}$ and define in a recurrent way the polynomials $g_0,\ldots,g_h$ by putting $g_0=y$, $g_1=g_0^{n_1}+c_1x^{a_{10}}=y^{n/e_1}+c_1
x^{\overline{b_1}/e_1}$,$\ldots,$ $g_h=g_{h-1}^{n_h}+c_hx^{a_{h0}}g_0^{a_{h1}}\cdots g_{h-2}^{a_{hh-1}}$.

\begin{teorema}
\label{T-R}
(cf. \cite{Teissier} and \cite{Reguera})
The polynomials $g_0,\ldots,g_h$ are distinguished and irreducible. We
have $\overline{\hbox{\rm char}_x} g_k=\left(\frac{\overline{b_0}}{e_k},\ldots,\frac{\overline{b_h}}{e_k}\right).$ The sequence $g_0,\ldots g_{k-1}$ is a sequence of key polynomials of $g_k$.
\end{teorema}

\noindent \begin{proof}
The theorem follows from Theorem \ref{caract} by induction on $k$.
\end{proof}

\medskip

\noindent {\bf Notes} 

\noindent Theorem \ref{B-A}
 characterizing the semigroups associated with branches is due to Bresinsky \cite{Bresinsky} (the case of characteristic $0$) and to Angerm\"uller \cite{Angermuller} (the case of arbitrary characteristic, see also 
\cite{Garcia-Stohr}). Both authors consider only {\em generic} case, i.e. $i_0(f,x)=\ord f$. Theorem \ref{T-R} which gives an explicit equation of the branch with given semigroup was obtained by Teissier by the method of deformations of the monomial curve associated with a branch. Another proof was given by Reguera L\'opez in \cite{Reguera}.

\section{ Description of branches with given semigroup}
\label{given semigroup}
\noindent We need two preliminary lemmas.

\medskip

\begin{lema}
\label{lemm1}
Let $f\in \bK[[x]][y]$ be a distinguished irreducible polynomial of degree $n>0$. Suppose that $\overline{\hbox{\rm char}_x} f=(\overline{b_0},\ldots,\overline{b_h})$ , $\overline{b_0}=n$ and let $f_0,f_1,\ldots,f_{h-1}$ be a sequence of key polynomials of $f$. Then any polynomial $g\in \bK[[x]][y]$ of $y$-degree strictly less than $n$ has a unique expansion of the form
\[
g=\sum g_{\alpha_1,\ldots,\alpha_h}f_0^{\alpha_1}\cdots f_{h-1}^{\alpha_h}, \;\; g_{\alpha_1,\ldots,\alpha_h}\in \bK[[x]],
\]

\noindent where $0\leq \alpha_1< n_1$, ..., $0\leq \alpha_h< n_h$. Moreover
\begin{enumerate}
\item the $y$-degrees of the terms appearing in the right-hand side of the preceding equality are all distinct,
\item $i_0(f,g)=\inf \{(\ord  g_{\alpha_1,\ldots,\alpha_h})n+\alpha_1 \overline{b_1}+\cdots+\alpha_h\overline{b_h}\;:\;
0\leq \alpha_i<n_i\;\hbox{\rm for } i=1,\ldots,h \}$.
\end{enumerate}
\end{lema}

\noindent \begin{proof}
The existence and unicity of the expansion and the inequality for the degrees holds for polynomials with coefficients in arbitrary integral domain (see \cite{Abhyankar3}, Section 2). The formula for the intersection multiplicity follows from the observation that the intersection multiplicities $i_0(f,  g_{\alpha_1,\ldots,\alpha_h}f_0^{\alpha_1}\cdots f_{h-1}^{\alpha_h})=(\ord  g_{\alpha_1,\ldots,\alpha_h})n+\alpha_1 \overline{b_1}+\cdots+\alpha_h\overline{b_h}$ are pairwise distinct by the unicity of B\'ezout's relation. 
\end{proof}

\medskip

\begin{lema}
\label{lemm2}
Under the notation and assumptions introduced above, if $\deg_y g <  n/e_k$ then $i_0(f,g)=e_ki_0(f_k,g)$.
\end{lema}

\noindent \begin{proof}
Suppose that  $\deg_y g <  n/e_k$. Then by Lemma \ref{lemm1} we get
$$
g=\sum g_{\alpha_1,\ldots,\alpha_h}f_0^{\alpha_1}\cdots f_{h-1}^{\alpha_h}, \;\;
$$

\noindent where $0\leq \alpha_i< n_i$, for $i\in \{1,\ldots,h\}$. Since  $\deg_y g <  n/e_k$ we have, by the first statement of
Lemma \ref{lemm1}, $\alpha_{k+1}=\cdots=\alpha_h=0$.

\medskip

\noindent By Proposition \ref{p1} $f_k$ is an irreducible distinguished polynomial,  $\overline{\hbox{\rm char}_x} f_k=\left(\frac{\overline{b_0}}{e_k},\ldots,\frac{\overline{b_k}}{e_k}\right)$ and $f_0,\ldots,f_{k-1}$ are key polynomials of
$f_k$. Therefore there exist $\alpha_1,\ldots,\alpha_k$ such that
\[
i_0(f_k,g)=(\ord  g_{\alpha_1,\ldots,\alpha_k,0,\ldots,0})\frac{n}{e_k}+\alpha_1 \frac{ \overline{b_1}}{e_k}+\cdots+\alpha_k\frac{\overline{b_k}}{e_k}=\frac{1}{e_k}i_0(f,g)
\]

\noindent and the lemma follows.
\end{proof}

\medskip

\noindent Let $\phi,f\in \bK[[x]][y]$ be distinguished polynomials such that $N=\frac{\deg_y f}{\deg_y \phi}$ is a positive integer.
Consider the $\phi$-adic expansion of $f$:

$$f=\phi^N+\alpha_1\phi^{N-1}+\cdots+\alpha_{N},\;\;\deg_y \alpha_i<\deg_y \phi \,\,\hbox{\rm for } i\in \{1,\ldots,N\}.$$

\noindent Put $\alpha_0=1$ and $I=\{i\in [0,N]\;:\;i_0(\alpha_i,\phi)\neq +\infty\}$. We define the {\em Newton polygon} $ \Delta_{x,\phi}(f)$ of $f$ with respect to the pair $(x,\phi)$ by setting
$$\Delta_{x,\phi}(f)=\hbox{\rm convex}\bigcup_{i\in I}\left\{(i_0(\alpha_i,\phi),N-i)+\bR^2_{\geq 0}\right\}.$$

\noindent The polygon $\Delta_{x,\phi}(f)$ intersects the vertical axis in the point $(0,N)$ and the horizontal axis in the point
$(i_0(f,\phi),0)$ provided that $i_0(f,\phi)\neq +\infty$. If $\phi=y$ then $\Delta_{x,\phi}(f)=\Delta_{x,y}(f)$ is the usual Newton polygon of $f$ in coordinates $(x,y)$. 

\medskip

\noindent In the sequel we use Teissier's notation (see \cite{Teissier2}): for any integers $k,l>0$ we put

$$\Teiss{k}{l}=\hbox{\rm convex}\{((k,0)+\bR^2_{\geq 0})\cup ((0,l)+\bR^2_{\geq 0})\}.$$

\begin{prop}
\label{newton}
If $f\in \bK[[x,y]]$ is an irreducible power series, $\overline{\hbox{\rm char}_x} f=(\overline{b_0},\ldots,\overline{b_h})$  and $f_0,f_1,\ldots,f_{h}$ is a sequence of key polynomials of $f$ then

\[\Delta_{x,f_{k-1}}(f_k)=\Teiss{\overline{b_k}/e_k}{e_{k-1}/e_k}.\]
\end{prop}

\noindent \begin{proof}
The $f_{k-1}$-adic expansion of $f_k$ is of the form $f_k=f_{k-1}^{n_k}+a_1f_{k-1}^{n_{k-1}}+\cdots+a_{n_k}$, where
$\deg_y a_i<\frac{n}{e_{k-1}}$ for $i\in \{1,\ldots,n_k\}$. By Proposition \ref{adic-expansion} we have $i_0(f,a_i)>i\overline{b_k}$ for $0<i<n_k$ and $i_0(f,a_{n_k})=n_k\overline{b_k}$. By Lemma \ref{lemm2} we get $i_0(f,a_i)=e_{k-1}i_0(f_{k-1},a_i)$ for $0<i \leq n_k$. Therefore we have $i_0(f_{k-1},a_i)>i\frac{\overline{b_k}}{e_{k-1}}$ for $0<i<n_k$ and $i_0(f_{k-1},a_{n_k})=\frac{\overline{b_k}}{e_{k}}$, which implies $\Delta_{x,f_{k-1}}(f_k)=\Teiss{\overline{b_k}/e_k}{e_{k-1}/e_k}.$
\end{proof}

\begin{prop}
\label{newton2}
Let $f$ be an irreducible distinguished polynomial of degree $n>1$. Let  $\overline{\hbox{\rm char}_x} f=(\overline{b_0},\ldots,\overline{b_h})$  and let $\phi\in \bK[[x]][y]$ be a $(h-1)$-key polynomial of $f$. Then
\begin{enumerate}
\item $\overline{\hbox{\rm char}_x} \phi=\left(\frac{\overline{b_0}}{e_{h-1}},\ldots,\frac{\overline{b_{h-1}}}{e_{h-1}}\right)$,
\item $\Delta_{x,\phi}(f)=\Teiss{\overline{b_h}}{e_{h-1}}$.
\end{enumerate}
\end{prop}

\noindent \begin{proof} The proposition follows from Propositions \ref{p1} and \ref{newton}.
\end{proof}

\medskip

\noindent The following theorem is the main result of this section.

\begin{teorema}
\label{newtont}
Let $(\overline{b_0},\ldots,\overline{b_h})$ be an $n$-characteristic sequence ($n>1$). Let $f\in \bK[[x]][y]$ be a distinguished polynomial of degree $n$ for which there exists an irreducible distinguished polynomial $\phi\in \bK[[x]][y]$ such that
\begin{enumerate}
\item $\overline{\hbox{\rm char}_x} \phi=\left(\frac{\overline{b_0}}{e_{h-1}},\ldots,\frac{\overline{b_{h-1}}}{e_{h-1}}\right)$,
\item $\Delta_{x,\phi}(f)=\Teiss{\overline{b_h}}{e_{h-1}}$.
\end{enumerate}

\noindent Then $f$ is irreducible, $\overline{\hbox{\rm char}_x} f=(\overline{b_0},\ldots,\overline{b_h})$ and $\phi$ is a key polynomial of $f$.
\end{teorema}

\noindent \begin{proof}
Let $\phi_0,\ldots,\phi_{h-2},\phi_{h-1}=\phi$ be a sequence of key polynomials of $\phi$.

\medskip
\noindent From the assumption about $\Delta_{x,\phi}(f)$ it follows that if $f=\phi^{n_h}+\alpha_1 \phi^{n_h-1}+\cdots+\alpha_{n_h}$, $\deg_y \alpha_i< \frac{n}{e_{h-1}}$ for $i\in\{1,\ldots,n_h\}$ is the $\phi$-adic expansion of $f$ then

$$(1)\;\;i_0(f,\phi)=i_0(\alpha_{n_h},\phi)=\overline{b_h},$$

\noindent  and

$$(2)\;\; i_0(\alpha_i,\phi)>i\frac{\overline{b_h}}{n_h}\;\; \hbox{\rm for } 0<i<n_h$$

\noindent (note that $\gcd(n_h,\overline{b_h})=\gcd(e_{h-1},\overline{b_h})=e_h=1$ whence the strict inequality in $(2)$).

\medskip

\noindent There exists a unique sentence of integers $l_0,\ldots,l_{h-1}$ such that $l_0\overline{b_0}+\cdots+l_{h-1}\overline{b_{h-1}}=e_{h-1}\overline{b_h}$,  where $l_0>0$ and $0\leq l_i<n_i$ for
$i\in\{1,\ldots,h-1\}$. Therefore we have $i_0(\phi,\alpha_{n_h})=\overline{b_h}=i_0(\phi, x^{l_0}\phi_0^{l_1}\cdots\phi_{h-2}^{l_{h-1}})$. Let $c\in \bK$ be a constant such that $i_0(\phi,\alpha_{n_h}-cx^{l_0}\phi_0^{l_1}\cdots\phi_{h-2}^{l_{h-1}})>i_0(\phi,\alpha_{n_h})=\overline{b_h}$. Put $\tilde{f}=\phi^{n_h}+cx^{l_0}\phi_0^{l_1}\cdots\phi_{h-2}^{l_{h-1}}$. Then by Theorem \ref{caract} $\tilde f\in \bK[[x]][y]$ is an irreducible distinguished polynomial of degree $n$, $\overline{\hbox{\rm char}_x} \tilde f=(\overline{b_0},\ldots,\overline{b_h})$ and $\phi$ is a key polynomial of degree $\frac{n}{e_{h-1}}$ of $\tilde f$.

\noindent We have $i_0(f,x)=i_0(\tilde f,x)=n$. Let $\tilde \alpha_{n_h}=\alpha_{n_h}-cx^{l_0}\phi_0^{l_1}\cdots\phi_{h-2}^{l_{h-1}}$ and consider

\begin{eqnarray*}(3)\;\; i_0(\tilde f,f)&=&i_0(\phi^{n_h}+cx^{l_0}\phi_0^{l_1}\cdots\phi_{h-2}^{l_{h-1}},\phi^{n_h}+\alpha_1 \phi^{n_h-1}+\cdots+\alpha_{n_h})\\
&=&i_0(\tilde f, \alpha_1 \phi^{n_h-1}+\cdots+\alpha_{n_{h-1}}\phi+\tilde{\alpha}_{n_h})\\
&\geq & \inf \{i_0(\tilde f, \alpha_1\phi^{n_h-1}),\ldots,i_0(\tilde f, \alpha_{n_h-1}\phi),i_0(\tilde f, \tilde {\alpha}_{n_h})\}
\end{eqnarray*}

\noindent since  $\tilde f$ is irreducible. Fix $i\in \{1,\ldots,n_h-1\}$. Then 
$$(4)\;\;i_0(\tilde f, \alpha_i \phi^{n_h-i})=i_0(\tilde f, \alpha_i)+(n_{ h}-i)i_0(\tilde f, \phi)=e_{h-1}i_0(\phi, \alpha_i)+(n_h-i)\overline{b_h}$$

\noindent since $\phi$ is a $(h-1)$-th key polynomial of $\tilde f$ and $i_0(\tilde f, \alpha_i)=e_{h-1}i_0(\phi,\alpha_i)$ by Lemma \ref{lemm2}. Using $(2)$ and $(3)$ we get

$$(5)\;\;i_0(\tilde f, \alpha_i \phi^{n_h-i})>e_{h-1}i\frac{\overline{b_h}}{e_{h-1}}+(n_h-i)\overline{b_h}=n_h\overline{b_h}$$

\noindent for $0<i<n_h$. Moreover, again by Lemma \ref{lemm2}

$$(6)\;\;i_0(\tilde f, \tilde{\alpha}_{n_h})=e_{h-1}i_0(\phi, \tilde {\alpha}_{n_h})>e_{h-1}i_0(\phi,\alpha_{n_h})=e_{h-1}\overline{b_h}.$$

\noindent Using $(3)$,  $(5)$ and $(6)$ we obtain $i_0(\tilde f, f)>e_{h-1}\overline{b_h}$  and the theorem follows from
  the Abhyankar-Moh irreducibility criterion (Corollary \ref{coro-AM-irr}).
\end{proof}

\medskip

\noindent Using Proposition \ref{newton2} and Theorem \ref{newtont} we get a recurrent description of the class of branches with given semigroup.

\begin{teorema}
Let  $(\overline{b_0},\ldots,\overline{b_h})$ be an $n$-characteristic sequence ($n>1$) and let $f\in \bK[[x]][y]$ be a distinguished polynomial of degree $n$. Then the following two conditions are equivalent
\begin{enumerate}
\item $f$ is irreducible and $\overline{\hbox{\rm char}_x} f=(\overline{b_0},\ldots,\overline{b_h})$,

\item there exists a distinguished irreducible polynomial $\phi \in \bK [[x]][y]$ such that
\begin{enumerate}
\item $\overline{\hbox{\rm char}_x} \phi=\left(\frac{\overline{b_0}}{e_{h-1}},\ldots,\frac{\overline{b_{h-1}}}{e_{h-1}}\right)$,
\item $\Delta_{x,\phi}(f)=\Teiss{\overline{b_h}}{e_{h-1}}$.
\end{enumerate}
\end{enumerate}
\end{teorema}

\medskip

\noindent To illustrate the above result let us write down

\begin{coro}
\label{newtonc}
Let $f\in \bK[[x]][y]$ be a distinguished polynomial of degree $n>1$ and let $m>0$ be an integer such that $\gcd(n,m)=1$. Then $f$ is irreducible with $\overline{\hbox{\rm char}_x} f=(n,m)$ if and only if there exists a power series $\psi(x)\in \bK[[x]]$, $\psi(0)=0$ such that
$$f=(y+\psi(x))^n+\alpha_1(x)(y+\psi(x))^{n-1}+\cdots+\alpha_n(x),$$

\noindent where $\ord \alpha_i > i \frac{m}{n}$ for $0<i<n$  and $\ord \alpha_n=m$.
\end{coro}

\begin{ejemplo}[see \cite{Teissier3}, Example 4.23]
Let $p=\hbox{\rm char} \bK\neq 0$. Let $f=y^p-x^{p-1}(1+y)=y^p-x^{p-1}y-x^{p-1}$. Using Corollary \ref{newtonc} with
$\psi(x)\equiv 0$ we check that $f$ is irreducible and $\overline{\hbox{\rm char}_x} f=(p,p-1)$. 
\end{ejemplo}

\begin{teorema}[Abhyankar's irreducibility criterion]
\label{Acriterion}
Let $f\in \bK[[x]][y]$ be a distinguished polynomial of degree $n>1$. Assume that $n\not\equiv 0$ (mod char $\bK$). Then 
$f$ is irreducible if and only if there exists an $n$-characteristic sequence $\overline{b_0},\ldots,\overline{b_h}$ such that
\begin{enumerate}
\item $i_0(f,\sqrt[e_{k-1}]{f})=\overline{b_k} $ and
\item $\Delta_{x,\sqrt[e_{k-1}]{f}}(\sqrt[e_{k}]{f})=\Teiss{\overline{b_k}/e_k}{e_{k-1}/e_k}$  for $k\in\{1,\ldots,h\}.$
\end{enumerate}
\end{teorema}

\noindent  \begin{proof}
The conditions are necessary: if $f$ is irreducible and  $\overline{\hbox{\rm char}_x} f=(\overline{b_0},\ldots,\overline{b_h})$
then both statements hold by Theorem \ref{Abhyankar-Moh} and Proposition \ref{newton}.

\medskip

\noindent The conditions are sufficient: this assertion follows from Theorem \ref{newtont} by induction on the length $h$ of the $n$-characteristic sequence.
\end{proof}

\medskip

\noindent To check the first condition we determine the sequences  $\overline{b_0},\ldots,\overline{b_h}$ and $e_0,\ldots,e_h$ such that

\begin{itemize}
\item $\overline{b_0}=b_0=n$,
\item $\overline{b_k}=i_0(f,\sqrt[e_{k-1}]{f})$, $e_k=\gcd(e_{k-1},\overline{b_k})$ for $k\in \{1,\ldots,h\},$
\item $e_0>\cdots>e_h=\gcd (e_h, i_0(f,\sqrt[e_{h}]{f}))$.
\end{itemize}

\noindent The first condition holds if and only if $e_h=1$ and $n_{k-1}\overline{b_{k-1}}<\overline{b_k}$ for $k>1$.

\medskip

\begin{ejemplo} (see \cite{Popescu}, p. 301)
Let $f=(y^2-x^3)^2-4x^5y-x^7$. Assume that char $\bK\neq 2$. Let $\overline{b_0}=e_0=4$, $\overline{b_1}=i_0(f,\sqrt[1]{f})=i_0(f,y)=6$, $e_1=\gcd(4,6)=2$, $\overline{b_1}=i_0(f,\sqrt[2]{f})=i_0(f,y^2-x^3)=13$, $e_2=\gcd(2,13)=1$. The sequence $(\overline{b_0},\overline{b_1},\overline{b_2})=(4,6,13)$ is a  $4$-characteristic since $n_1\overline{b_1}=12$ is strictly less than $\overline{b_2}=13$. Thus condition (1) in Abhyankar's irreducibility criterion holds. To check (2) we compute

\[ \Delta_{x,\sqrt[4]{f}}(\sqrt[2]{f})=\Delta_{x,y}(y^2-x^3)=\Teiss{3}{2}=\Teiss{\overline{b_1}/e_1}{e_{0}/e_1} \]

\noindent and

\[\Delta_{x,\sqrt[2]{f}}(\sqrt[1]{f})=\Delta_{x,\sqrt[2]{f}}((\sqrt[2]{f})^2-4x^5y-x^7)=\Teiss{13}{2}=\Teiss{\overline{b_2}/e_2}{e_{1}/e_2}.\]

\noindent Therefore condition $(2)$ holds and by Theorem \ref{Acriterion} $f$ is irreducible provided that char $\bK\neq 2$. If
char $\bK=2$ then $f=y^4+x^6-x^7=(y^2+ix^6+\cdots)(y^2-ix^6+\cdots)$,  where $i^2=-1$ in $\bK$ and $f$ is not irreducible.

\end{ejemplo}
\medskip

\noindent {\bf Notes} 

\noindent The first description of the class of branches with given semigroup is due to Teissier \cite{Teissier} (see also \cite{Cassou} and \cite{Jaworski}). Our approach is inspired by the papers by Abhyankar \cite{Abhyankar2} and Kuo\cite{Kuo} (see also \cite{McCallum} and \cite{Assi2}). The generalization of the Newton polygon introduced by Kuo in \cite{Kuo} is useful in Valuation Theory \cite{Vaquie}, Section 5. Our presentation of Abhyankar's irreducibility criterion differs from the original one. Another version of Abhyankar's criterion is due to Cossart and Moreno-Soc\'{\i}as \cite{Cossart1}  and \cite{Cossart2} . A criterion of irreducibility based on different ideas was given recently by \cite{GB-G}. The $g$-adic expansions of polynomials and Newton polygons were applied to generalize the classical  Sh\"onemann-Eisenstein irreducibility criterion in the early  twentieth century (see \cite{Ore}).

\section{Merle-Granja's Factorization Theorem}
\label{Merle-Granja}
\noindent Let $\{f=0\}$ be a branch different from $\{x=0\}$. Let 
$\overline{\hbox{\rm char}_x} f=(\overline{b_0},\ldots,\overline{b_h})$ , $\overline{b_0}=n>1$. In this section we prove the following result on factorization of power series (see \cite{Granja} and \cite{Merle}).

\begin{teorema}[Merle-Granja's Factorization Theorem]
\label{factorization}
Fix $k$, $1\leq k\leq h$. Let $g=g(x,y)\in \bK[[x,y]]$ be a power series such that
\begin{enumerate}
\item $i_0(g,x)=\frac{n}{e_k}-1$,
\item $i_0(f,g)=\sum_{i=1}^k (n_i-1)\overline{b_i}$.
\end{enumerate}

\noindent Then there is a factorization $g=g_1\cdots g_k\in \bK[[x,y]]$ such that

\begin{enumerate}
\item $i_0(g_i,x)=\frac{n}{e_i}-\frac{n}{e_{i-1}}$ for $i\in \{1,\ldots,k\}$,
\item if $\phi\in \bK[[x,y]]$ is an irreducible factor of $g_i$, $i\in\{1,\ldots,k\}$ then 
\begin{enumerate}
\item $\frac{i_0(f,\phi)}{i_0(\phi,x)}=\frac{e_{i-1}\overline{b_i}}{n}$,
\item $i_0(\phi,x)\equiv 0$  $($mod $\frac{n}{e_{i-1}})$.
\end{enumerate}
\end{enumerate}
\end{teorema}

\noindent This statement  is very close to Granja's theorem (see \cite{Granja}) where the Ap\'ery sequences are used) and is a generalization of Merle's result on polar curves. 

\medskip

\begin{nota}
Let $f_0,\ldots,f_h$ be a sequence of key polynomials of $f$. Fix $k\in \{1,\ldots,h\}$. Take $g=f_0^{n_1-1}\cdots f_{k-1}^{n_k-1}$. Then $g$ satisfies conditions $(1)$ and $(2)$ of Theorem \ref{factorization} . Here $g_i=f_{i-1}^{n_i-1}$ for $i\in\{1,\ldots,k\}$.

\medskip

\noindent If $\phi$ is an irreducible factor of $g_i$ then $\phi=f_{i-1} \cdot $ unit. Clearly $\frac{i_0(f,\phi)}{i_0(\phi,x)}=\frac{e_{i-1}\overline{b_i}}{n}$ and $i_0(\phi,x)=\frac{n}{e_{i-1}}$.
\end{nota}

\medskip

\noindent To prove Theorem \ref{factorization} we need a few lemmas.

\begin{lema}
\label{le1}
Let $k$ be an integer such that $1\leq k\leq h$. Then

\begin{enumerate}
\item $(n_1-1)\overline{b_1}+\cdots+(n_k-1)\overline{b_k}\not\equiv 0$ $($mod $e_{k-1})$,
\item $(n_1-1)\overline{b_1}+\cdots+(n_k-1)\overline{b_k}<
\overline{b_{k+1}}$,
\item if $(n_1-1)\overline{b_1}+\cdots+(n_k-1)\overline{b_k}=a_0\overline{b_0}+
a_1\overline{b_1}+\cdots+a_k\overline{b_k}$ with integers $a_0,\ldots,a_k$ such that $a_0,a_k \geq 0$ and $0\leq a_i<n_i$ for $i\in\{1,\ldots,k-1\}$ then $a_0=0$ and $a_i=n_i-1$ for $i\in\{1,\ldots,k\}$.
\end{enumerate}
\end{lema}

\noindent \begin{proof}
Suppose that $(n_1-1)\overline{b_1}+\cdots+(n_k-1)\overline{b_k}\equiv 0$ (mod $e_{k-1}$). Then $(n_k-1)\overline{b_k}\equiv 0$ (mod $e_{k-1}$) and $(n_k-1)\frac{\overline{b_k}}{e_k}\equiv 0$ (mod $n_k$). We get a
contradiction because $\gcd(\frac{\overline{b_k}}{e_k},n_k)=1$.
\medskip

 \noindent Now we will prove the second statement: if $k=h$ then $\overline{b_{k+1}}=\overline{b_{h+1}}=+\infty$ and the inequality is obvious. Let $k<h$. Then $(n_1-1)\overline{b_1}+\cdots+(n_k-1)\overline{b_k}=
(n_1\overline{b_1}-\overline{b_1})+\cdots+(n_k\overline{b_k}-
\overline{b_k})<\overline{b_{k+1}}-\overline{b_1}<\overline{b_{k+1}}$.

\medskip
\noindent To finish the proof, let $a_k=qn_k+a'_k$ with $0\leq a'_k<n_k$. By assumption we get $qn_k\overline{b_k}=(-a_0)\overline{b_0}+(n_1-1-a_1)\overline{b_1}+\cdots+(n_k-1-a'_k)\overline{b_k}$. The identity above is a B\'ezout's relation. Since $qn_k\overline{b_k}\equiv 0$ (mod $e_{k-1}$) we get by the unicity of B\'ezout's relation $-a_0 \geq 0$ and $n_k-1-a'_k=0$ that is $a_0=0$ and
$a'_k=n_k-1$. Therefore we get $qn_k\overline{b_k}=(n_1-1-a_1)\overline{b_1}+\cdots+(n_{k-1}-1)\overline{b_{k-1}}\leq (n_1-1)\overline{b_1}+\cdots+(n_{k-1}-1)\overline{b_{k-1}}<\overline{b_k}$ by the second statement. Thus $q=0$ and again by the unicity of B\'ezout's relation we get the last statement.
\end{proof}

\begin{lema}
\label{le2}
Let $\phi\in \bK[[x,y]]$ be an irreducible power series such that $\frac{i_0(f,\phi)}{i_0(\phi,x)}<\frac{e_{k-1}\overline{b_k}}{n}$ for a $k>0$. Then $i_0(f,\phi)\in \overline{b_0}\bN+\cdots+\overline{b_{k-1}}\bN$.
\end{lema}

\noindent \begin{proof}
Let $f_{k-1}$ be a  $(k-1)$-th key polynomial of $f$. Thus $i_0(f_{k-1},x)=\frac{n}{e_{k-1}}$, $i_0(f,f_{k-1})=\overline{b_k}$ and $\overline{\hbox{\rm char}_x} f_{k-1}=\left(\frac{\overline{b_0}}{e_{k-1}},\ldots,\frac{\overline{b_{k-1}}}{e_{k-1}}\right)$. Since $\frac{e_{k-1}\overline{b_k}}
{n}=\frac{i_0(f,f_{k-1})}{i_0(f_{k-1},x)}$ we get by  assumption the inequality  $\frac{i_0(f,\phi)}
{i_0(\phi,x)}<\frac{i_0(f,f_{k-1})}{i_0(f_{k-1},x)}$ that is $d_x(f,\phi)<d_x(f,f_{k-1})$. Using the STI to the power series $\phi,f_{k-1}$ and $f$ we get $d_x(f,\phi)=d_x(\phi,f_{k-1})$, which implies $i_0(f,\phi)=\frac{i_0(f,x)}{i_0(f_{k-1},x)}i_0(f_{k-1},\phi)=e_{k-1}i_0(f_{k-1},\phi)\in \overline{b_0}\bN+\cdots+\overline{b_{k-1}}\bN$.
\end{proof}

\begin{lema}
\label{le3}
Let $\phi$ be an irreducible power series such that $\frac{i_0(f,\phi)}{i_0(\phi,x)}=\frac{e_{k-1}\overline{b_k}}{n}$ for a $k>0$. Then $i_0(\phi,x)\equiv 0$  $\left(mod\;\frac{n}{e_{k-1}}\right)$ and $i_0(f,\phi)\equiv 0$ $($mod $\overline{b_k})$.
\end{lema}

\noindent \begin{proof}
Since $i_0(f,\phi)=\frac{e_{k-1}\overline{b_k}}{n}i_0(\phi,x)$ it suffices to check that $i_0(\phi,x)\equiv 0$  $\left(mod \;\frac{n}{e_{k-1}}\right)$. If $k=1$ it is clear, so assume $k>1$. We have $\frac{i_0(f,\phi)}{i_0(\phi,x)}=\frac{e_{k-1}\overline{b_k}}{n}> \frac{e_{k-2}\overline{b_{k-1}}}{n}$ hence $i_0(\phi,x)\equiv 0$ $\left(mod \;\frac{n}{e_k}\right)$  by Lemma \ref{igual caracteristica}.
\end{proof}

\medskip

\noindent Now we can prove the Factorization Theorem.

\medskip

\noindent {\bf Proof of Theorem \ref{factorization}}
\noindent Let us fix a $k\in\{1,\ldots,h\}$ and let $g\in \bK[[x,y]]$ be such that the conditions $(1)$ and $(2)$ hold. Let $g=\phi_1\cdots \phi_s$ with irreducible $\phi_j\in \bK[[x,y]]$ for $j\in \{1,\ldots,s\}$. Firstly we check

\medskip

\noindent $(*)\;$  if $\phi$ is an irreducible factor of $g$ then
$\frac{i_0(f,\phi)}{i_0(\phi,x)}\leq \frac{e_{k-1}\overline{b_k}}{n}.$

\medskip
\noindent Indeed, in the contrary case there would exist an irreducible factor $\phi$ of $g$ such that $\frac{i_0(f,\phi)}{i_0(\phi,x)}> \frac{e_{k-1}\overline{b_k}}{n}$ and we would get $i_0(\phi,x)\equiv 0$ (mod $\frac{n}{e_k}$) by Lemma \ref{igual caracteristica},  which is a contradiction since $i_0(\phi,x)\leq i_0(g,x)\leq \frac{n}{e_k}-1<\frac{n}{e_k}$.

\medskip
\noindent $(**)\;$ There exists (at least one) irreducible factor $\phi$ of $g$ such that $\frac{i_0(f,\phi)}{i_0(\phi,x)}=\frac{e_{k-1}\overline{b_k}}{n}.$

\medskip

\noindent If $\frac{i_0(f,\phi)}{i_0(\phi,x)}\neq \frac{e_{k-1}\overline{b_k}}{n}$ for all irreducible factors of $g$ then we would get by $(*)$ $\frac{i_0(f,\phi_j)}{i_0(\phi_j,x)} < \frac{e_{k-1}\overline{b_k}}{n}$ for all $j\in \{1,\ldots,s\}.$

\medskip

\noindent By Lemma \ref{le2} we would have $i_0(f,\phi_j)\in \overline{b_0}\bN+\cdots+\overline{b_{k-1}}\bN$ for $j\in\{1,\ldots,s\}$ and consequently $i_0(f,g)=\sum_{j=1}^s i_0(f,\phi_j)\in \overline{b_0}\bN+\cdots+\overline{b_{k-1}}\bN$. This is impossible because $i_0(f,g)=\sum_{i=1}^n (n_i-1)\overline{b_i}\not\equiv 0$ (mod $e_{k-1}$) by the first statement of  Lemma 
\ref{le1}.

\medskip

\noindent Now, let us put $g_k$ the product of all factors $\phi_j$ of $g$ such that $\frac{i_0(f,\phi_j)}{i_0(\phi_j,x)}=\frac{e_{k-1}\overline{b_k}}{n}$. Therefore we get $g=\tilde g g_k$ in $\bK[[x,y]]$. Using Lemma \ref{le2} we check that $i_0(f,\tilde g)\in \overline{b_0}\bN+\cdots+\overline{b_{k-1}}\bN$ and by Lemma \ref{le3} we get $i_0(f,g_k)\equiv 0$ (mod $\overline{b_k})$.

\medskip

\noindent Let us write $i_0(f,\tilde g)=a_0\overline{b_0}+a_1\overline{b_1}+\cdots+a_{k-1}\overline{b_{k-1}}$ with $a_0\geq 0$ and $0\leq a_i\leq n_i-1$ and $i_0(f,g_k)=a_k\overline{b_k}$, $a_k\geq 0$. Therefore we get $\sum_{i=1}^k(n_i-1)\overline{b_i}=i_0(f,g)=i_0(f,\tilde g)+i_0(f,g_k)=a_0\overline{b_0}+a_1\overline{b_1}+
\cdots+a_{k-1}\overline{b_{k-1}}+a_k\overline{b_k}$. 

\medskip

\noindent By the third statement of Lemma \ref{le1} we have $a_0=0$ and $a_k=n_k-1$. Thus $i_0(f,g_k)=(n_k-1)\overline{b_k}$ and $i_0(g_k,x)=\frac{n}{e_k}-\frac{n}{e_{k-1}}$ since $\frac{i_0(f,g_k)}{i_0(g_k,x)}=\frac{e_{k-1}\overline{b_k}}{n}$.

\medskip

\noindent If $k=1$ we are done (if $k=1$ then $i_0(f,\tilde g)=0$, that is $\tilde g$ is a unit and we put $\tilde g g_1$ instead of $g_1$). If $k> 1$ then $\tilde g=\frac{g}{g_k}$ satisfies the assumptions of Theorem \ref{factorization} with $k-1$. We use induction on $k$.
$\blacksquare$

\begin{nota}
In the proof of Merle-Granja's factorization theorem we used the inequality $i_0(g,x)\leq \frac{n}{e_k}-1$ instead of the equality 
$i_0(g,x)=\frac{n}{e_k}-1$. Therefore this inequality and condition (2) imply condition (1) of Theorem \ref{factorization}. 
\end{nota}

\begin{teorema}[Merle's factorization theorem]
Suppose that  $\overline{\hbox
{\rm char}_x} f=(\overline{b_0},\ldots,\overline{b_h})$,
$\overline{b_0}=n> 1$ and $n\not\equiv 0$ (mod char $\bK$). Then $\frac{\partial f}{\partial y}=g_1\cdots g_h$ in $\bK[[x,y]],$ where

\begin{enumerate}
\item $i_0(g_i,x)=\frac{n}{e_i}-\frac{n}{e_{i-1}}$ for $i\in \{1,\ldots,k\}$.
\item If $\phi\in \bK[[x,y]]$ is an irreducible factor of $g_i$, $i\in \{1,\ldots,h\}$, then $\frac{i_0(f,\phi)}{i_0(\phi,x)}=\frac{e_{i-1}\overline{b_i}}{n}$ and $i_0(\phi,x)\equiv 0$  $\left(mod \;\frac{n}{e_{i-1}}\right)$.
\end{enumerate}
\end{teorema}

\noindent \begin{proof}
Since $n\not\equiv 0$ (mod char $\bK$) we have $i_0\left(\frac{\partial f}{\partial y},x\right)=n-1$. By the Dedekind formula and
the Conductor formula we have $i_0\left(f,\frac{\partial f}{\partial y}\right)=c(f)+n-1=\sum_{k=1}^h (n_k-1)\overline{b_k}$. We apply 
Theorem \ref{factorization} to the series $g=\frac{\partial f}{\partial y }$.
\end{proof}

\medskip

\noindent {\bf Notes} 

\noindent The first result on factorization of the derivative was proved by Henry J.S. Smith in \cite{Smith} but his work fell into
oblivion for a long time. Merle proved the factorization theorem in the generic case, the observation that the theorem is true in any coordinates is due to Ephraim \cite{Ephraim}. Granja's theorem is formulated in terms of Ap\'ery sequences and proved using the Hamburger-Noether expansions (see \cite{Granja}).

\medskip
\noindent {\bf Acknowledgements:} The authors are very grateful to Bernard Teissier for reading the manuscript and making valuable suggestions.

\medskip
\noindent
{\small Evelia Rosa Garc\'{\i}a Barroso\\
Departamento de Matem\'atica Fundamental\\
Facultad de Matem\'aticas, Universidad de La Laguna\\
38271 La Laguna, Tenerife, Espa\~na\\
e-mail: ergarcia@ull.es}

\medskip

\noindent {\small Arkadiusz P\l oski\\
Department of Mathematics\\
Technical University \\
Al. 1000 L PP7\\
25-314 Kielce, Poland\\
e-mail: matap@tu.kielce.pl}
\end{document}